\documentclass[sts]{imsart}

\usepackage{amsmath, amssymb, color}
\usepackage{amsthm} 

\usepackage[algo2e]{algorithm2e} 
\usepackage{algpseudocode}
 \usepackage{natbib}
\usepackage[colorlinks,citecolor=blue,urlcolor=blue]{hyperref}
\usepackage{xcolor,colortbl}
\usepackage{caption}

\RequirePackage{hypernat}
\usepackage{cleveref}
\usepackage{color}
\usepackage{amsfonts, fancybox}
\usepackage{amssymb}
\usepackage[american]{babel}
\usepackage{psfrag,color}
\usepackage{dcolumn}
\usepackage[format=hang, justification=justified, singlelinecheck=false]{subcaption}
\usepackage{flafter, bm, dsfont}

\usepackage{multirow}
\usepackage{color}
\usepackage{amsmath}
\usepackage{float}
\usepackage{mathrsfs}
\usepackage{amsfonts}
\usepackage{dsfont}
\usepackage{amssymb}
\usepackage{wrapfig}

\usepackage{amssymb, latexsym}

\usepackage{graphicx}
\usepackage{epstopdf}
\usepackage{tikz}
\usepackage{pgfplots}
\usepgfplotslibrary{fillbetween}

\startlocaldefs

\pgfplotsset{
  every axis/.append style = {thick},tick style = {thick,black},
  /tikz/normal shift/.code 2 args = {%
    \pgftransformshift{%
        \pgfpointscale{#2}{\pgfplotspointouternormalvectorofticklabelaxis{#1}}%
    }%
  },%
  range3frame/.style = {
    tick align        = outside,
    scaled ticks      = false,
    enlargelimits     = false,
    ticklabel shift   = {10pt},
    axis lines*       = left,
    line cap          = round,
    clip              = false,
    xtick style       = {normal shift={x}{10pt}},
    ytick style       = {normal shift={y}{10pt}},
    ztick style       = {normal shift={z}{10pt}},
    x axis line style = {normal shift={x}{10pt}},
    y axis line style = {normal shift={y}{10pt}},
    z axis line style = {normal shift={z}{10pt}},
  }
}

\newcommand{\eps}{\varepsilon}

\newcommand{\tr}{\mathbf{Tr}}

\numberwithin{equation}{section}

\newtheorem{definition}{Definition}
\newtheorem{theorem}[definition]{Theorem}
\newtheorem{lemma}[definition]{Lemma}
\newtheorem{corollary}[definition]{Corollary}
\newtheorem{remark}[definition]{Remark}
\newtheorem{proposition}[definition]{Proposition}
\newtheorem{conjecture}[definition]{Conjecture}

\newtheorem{assumption}[definition]{Assumption}

\newcommand{\cB}{\mathcal{B}}

\newcommand{\cP}{\mathcal{P}}

\DeclareMathOperator*{\argmin}{argmin}

\newcommand{\fP}{\mathbf{P}}
\newcommand{\R}{\mathbf{R}}
\newcommand{\fS}{\mathbf{S}}
\newcommand{\bX}{\mathbb{X}}
\newcommand{\bD}{\mathbb{D}}
\newcommand{\vect}{\textbf{vec}}

\providecommand{\NB}{\color{red}$\spadesuit$}
\newcommand{\T}{\mathbb{T}}
\newcommand{\pen}{p}
\def\d{{\rm d}}
\def\Var{\operatorname{Var}}
\newcommand{\Ts}{\Theta_\star}
\newcommand{\Ind}{{\bf 1}}
\newcommand{\Lm}{\mathcal{L}}
\newcommand{\Th}{\hat\Theta}
\newcommand{\Sigh}{\hat{\Sigma}}
\newcommand{\Sigs}{\Sigma_\star}
\newcommand{\ex}{e}
\newcommand{\bigO}{\mathcal{O}}
\DeclareMathOperator{\rank}{\mathbf{rank}}
\newcommand{\NP}{\mathsf{NP}}
\newcommand{\G}{\mathcal{G}}
\renewcommand{\P}{\mathbf{P}}

\newcommand{\Kc}{{\sf KL}}
\newcommand{\Lc}{\ell_Y}
\makeatletter
\newsavebox{\@brx}
\newcommand{\llangle}[1][]{\savebox{\@brx}{\(\m@th{#1\langle}\)}%
	\mathopen{\copy\@brx\mkern2mu\kern-0.9\wd\@brx\usebox{\@brx}}}
\newcommand{\rrangle}[1][]{\savebox{\@brx}{\(\m@th{#1\rangle}\)}%
	\mathclose{\copy\@brx\mkern2mu\kern-0.9\wd\@brx\usebox{\@brx}}}
\makeatother

\newcommand{\card}{\operatorname{\mathsf{card}}}
\newcommand{\h}{{\rm I}\kern-0.18em{\rm H}}
\newcommand{\K}{{\rm I}\kern-0.18em{\rm K}}
\newcommand{\p}{{\rm I}\kern-0.18em{\rm P}}
\newcommand{\E}{{\rm I}\kern-0.18em{\rm E}}
\newcommand{\Z}{{\rm Z}\kern-0.18em{\rm Z}}
\newcommand{\1}{{\rm 1}\kern-0.24em{\rm I}}
\newcommand{\N}{{\rm I}\kern-0.18em{\rm N}}

\definecolor{MIT}{RGB}{163,31,52}
\endlocaldefs

\crefname{theorem}{Theorem}{Theorems}
\crefname{assumption}{Assumption}{Assumptions}
\crefname{observation}{Observation}{Observations}
\crefname{claim}{Claim}{Claims}
\crefname{condition}{Condition}{Conditions}
\crefname{example}{Example}{Examples}
\crefname{fact}{Fact}{Facts}
\crefname{lemma}{Lemma}{Lemmas}
\crefname{corollary}{Corollary}{Corollaries}
\crefname{definition}{Definition}{Definitions}
\crefname{remark}{Remark}{Remarks}

\DeclareMathOperator{\rk}{\mathbf{rank}}

\makeatletter
\newcommand*{\defeq}{\mathrel{\rlap{%
                     \raisebox{0.3ex}{$\m@th\cdot$}}%
                     \raisebox{-0.3ex}{$\m@th\cdot$}}%
                    =}
\newcommand*{\eqdef}{=
  \mathrel{\rlap{%
      \raisebox{0.3ex}{$\m@th\cdot$}}%
    \raisebox{-0.3ex}{$\m@th\cdot$}}%
}
\makeatother

\usepackage{comment}

\begin{document}

\begin{frontmatter}

\title{Optimal Link Prediction \\ with Matrix Logistic Regression}
\runtitle{Optimal Link Prediction with Matrix Logistic Regression}

\begin{aug}

\author{\fnms{Nicolai}~\snm{Baldin}\thanksref{t1}\ead[label=baldin]{n.baldin@statslab.cam.ac.uk}}
\and
\author{\fnms{Quentin}~\snm{Berthet}\thanksref{t2}\ead[label=berthet]{q.berthet@statslab.cam.ac.uk}}

\affiliation{University of Cambridge}
\thankstext{t1}{Supported by the European Research Council (ERC) under the European Union's Horizon 2020 research and innovation programme (grant agreement No 647812)}
\thankstext{t2}{Supported by an Isaac Newton Trust Early Career Support Scheme and by The Alan Turing Institute under the EPSRC grant EP/N510129/1.}

\address{{Nicolai Baldin}\\
{Department of Pure Mathematics} \\
{and Mathematical Statistics,} \\
{University of Cambridge}\\
{Wilbeforce Road}\\
{Cambridge, CB3 0WB, UK}\\
\printead{baldin}
}

\address{{Quentin Berthet}\\
{Department of Pure Mathematics} \\
{and Mathematical Statistics,} \\
{University of Cambridge}\\
{Wilbeforce Road}\\
{Cambridge, CB3 0WB, UK}\\
\printead{berthet}
}

\runauthor{Baldin and Berthet}
\end{aug}

\begin{abstract}
We consider the problem of link prediction, based on partial observation of a large network, and on side information associated to its vertices. The generative model is formulated as a matrix logistic regression. The performance of the model is analysed in a high-dimensional regime
under a structural assumption. The minimax rate for the Frobenius-norm risk is established and a combinatorial estimator based on the penalised maximum likelihood approach is shown to achieve it. Furthermore, it is shown that this rate cannot be attained by any (randomised) algorithm computable in polynomial time under a computational complexity assumption. 

\end{abstract}

	\begin{keyword}[class=MSC]
	\kwd[Primary ]{62J02}
	\kwd{62C20}
	\kwd[; secondary ]{68Q17}
\end{keyword}

\begin{keyword}[class=KWD]
Link prediction, Logistic regression, 
Computational lower bounds, Planted clique
\end{keyword}


\end{frontmatter}

	\section*{Introduction}

In the field of network analysis, the task of {\em link prediction} consists in predicting the presence or absence of edges in a large graph, based on the observations of some of its edges, and on side information. Network analysis has become a growing inspiration for statistical problems. Indeed, one of the main characteristics of datasets in the modern scientific landscape is not only their growing size, but also their increasing complexity. Most phenomena now studied in the natural and social sciences concern not only isolated and independent variables, but also their interactions and connections. 

The fundamental problem of link prediction is therefore naturally linked with statistical estimation: the objective is to understand, through a generative model, why different vertices are connected or not, and to generalise these observations to the rest of the graph.

Most statistical problems based on graphs are unsupervised: the graph itself is the sole data, there is no side information, and the objective is to recover an unknown structure in the generative model. Examples include the planted clique problem \citep{Kuc95,AloKriSud98}, the stochastic block model \citep{HolLasLei83}---see~\cite{Abb17} for a recent survey of a very active line of work \citep{DecKrzMoo11,MosNeeSly13, Mas14, MosNeeSly15, AbbSan15,BanMoo16}, the Ising blockmodel \citep{BerRigSri16}, random geometric graphs -- see \cite{Pen03} for an introduction and \citep{DevGyoLug11,BubDinEld14} for recent developments in statistics, or metric-based learning \citep{Chen09, Bellet14asurvey} and ordinal embedings \citep{JaiJamNow16}.  




In supervised regression problems on the other hand,  the focus is on understanding a fundamental mechanism, formalized as the link between two variables. The objective is to learn how an explanatory variable $X$ allows to predict a response $Y$, i.e. to find the unknown function $f$ that best approximates the relationship $Y \approx f(X)$. This statistical framework is often applied to the observation of a phenomenon measured by $Y$ (e.g. of a natural or social nature), given known information $X$: the principle is to understand said phenomenon, to explain the relationship between the variables by estimating the function $f$  \citep{Holland81, Hoff2002}.

We follow this approach here: our goal is to learn how known characteristics of each agent (represented by a node) in the network induce a greater or smaller chance of connection, to understand the mechanism of formation of the graph. We propose a model for supervised link prediction, using the principle of regression for inference on graphs. For each vertex, we are given side information, a vector of observations $X \in \R^d$. Given observations $X_i, X_j$ about nodes $i$ and $j$ of a network, we aim to understand how these two explanatory variables are related to the probability of connection between the two corresponding vertices, such that $\fP(Y_{(i,j)} = 1) = f(X_i,X_j)$, by estimating $f$ within a high-dimensional class based on logistic regression. Besides this high-dimensional parametric modelling, various fully  non-parametric statistical frameworks were exploited in the literature, see, for example, \cite{wolfe2013nonparametric, gao2015rate}, for graphon estimation, \cite{PapaNIPS2016, Biau2006} for graph reconstruction and \cite{Bickel21068} for modularity analysis.

Link prediction can be useful in any application where data can be gathered about the nodes of a network. One of the most obvious motivations is in social networks, in order to model social interactions. With access to side information about each member of a social network, the objective is to understand the mechanisms of connection between members: shared interests, differences in artistic tastes or political opinion \citep{wasserman_faust_1994}. This can also be applied to citation networks, or in the natural sciences to  biological networks of interactions between molecules or proteins \citep{Yu2008, Madeira2004}. The key assumption in this model is that the network is a consequence of the information, but not necessarily based  on similarity: it is possible to model more complex interactions, e.g. where opposites attract. 

The focus on a high-dimensional setting is another aspect of this work that is also motivated by modern applications of statistics: data is often collected without discernment and the ambient dimension $d$ can be much larger than the sample size. This setting is common in regression problems: the underlying model is often actually very simple, to reflect the fact that only a small number of measured parameters are relevant to the problem at hand, and that the intrinsic dimension is much smaller. This is usually handled through an assumption on the rank, sparsity, or regularity of a parameter. Here this needs to be adapted to a model with two covariates (explanatory variables) and a structural assumption is made in order to reflect this nature of our problem.

We therefore decide to tackle  link prediction by modelling it as matrix logistic regression. We study a generative model for which $\fP(Y_{(i,j)}=1) = \sigma(X_i^\top \Theta_\star X_j)$, where $\sigma$ is the sigmoid function, and $\Theta_\star$ is the unknown matrix to estimate. It is a simple way to model how the variables {\em interact}, by a quadratic {\em affinity} function and a sigmoid function. In order to model realistic situations with partial observations, we assume that $Y_{(i,j)}$ is only observed for a subset of all the couples $(i,j)$, denoted by $\Omega$.


To convey the general idea of a simple dependency on $X_i$ and $X_j$, we make structural assumptions on the rank and sparsity of $\Theta_\star$. This reflects that the affinity $X_i^\top \Theta_\star X_j$ is a function of the projections $u_\ell^\top X$ for the vectors $X_i$ and $X_j$, for a small number of orthogonal  vectors, that have themselves a small number of non-zero coefficients (sparsity assumption).
In order to impose that the inverse problem is well-posed, we also make a restricted conditioning assumption on $\Theta_\star$, inspired by the restricted isometry property (RIP). These conditions are discussed in Section~\ref{SEC:desc}. 
We talk of link prediction as this is the legacy name but we focus almost entirely on the problem of estimating $\Theta_\star$.

The classical techniques of likelihood maximization can lead to computationally intractable optimization problems. We show that in this problem as well as others this is a fundamental difficulty, not a weakness of one particular estimation technique; statistical and computational complexities are intertwined.\\

{\bf Our contribution:} This work is organized in the following manner: We give a formal description of the problem in Section~\ref{SEC:desc}, as well as a discussion of our assumptions and links with related work. Section~\ref{SEC:algo} collects our  main statistical results. 	We propose an estimator \(\hat{\Theta}\) based on the penalised maximum likelihood approach and  analyse its performance in Section~\ref{SEC:penLogLoss} in terms of non-asymptotic rate of estimation. We show that it attains the minimax rate of estimation over  simultaneously
block-sparse and low-rank matrices $\Theta_\star$, but is not computationally tractable. 
In Section \ref{SEC:feedback}, 
we provide a convex relaxation of the problem which is in essence the \emph{Lasso} estimator applied to a vectorised version of the problem. 
The link prediction task is covered in Section~\ref{SEC:slow}.
A matching minimax lower bound for the rate of estimation is given in Section~\ref{MinimaxLowerBound}. 
Furthermore, we show in Section~\ref{SEC:complow} that the minimax rate cannot be attained by a (randomised) polynomial-time algorithm, and we identify a corresponding computational lower bound. 
The proof of this bound is based on a reduction scheme from the so-called \emph{ dense subgraph detection problem}. 
Technical proofs are deferred to the appendix.
Our findings are depicted in Figure~\ref{HB}.  \\
\begin{figure}[tp]
	\centering
	\begin{tikzpicture}
	\node  at (7.5,2.3) {\text{ \scriptsize Lasso} };
	\node  at (7.5,1.5) {\(\scalebox{1.2}{$  \frac{k^2}{N} \log(d)$} \)};
	\node at (5.5,1.5) {\( \scalebox{1.2}{$\frac{k^2}{N}$} \)};
	\draw[decoration={brace,raise=5pt},decorate]
	(0,2.5) -- node[above=6pt] {\text{\scriptsize computationally hard} } (5.5,2.5);
	\draw [-, thick] (0,2) -- (8,2);
	\draw [-, thick] (5.5,1.9) -- (5.5,2.4);
	\draw [-, thick] (7.5,1.9) -- (7.5,2.1);
	\draw [-, thick] (2.7,1.9) -- (2.7,2.1);
	\node  at (2.7,2.3) {\text{ \scriptsize minimax} };
	\node at (2.7,1.5) {\( \scalebox{1.2}{$ \frac{kr}{N} + \frac{k}{N}\log(\frac{d\ex}{k})$}\)};
	\end{tikzpicture}
	\caption{The computational and statistical boundaries for estimation and prediction in the matrix logistic regression model. Here \(k\) denotes the sparsity of \(\Ts\) and \(r\)  its rank, while  \(N\) is the number of observed edges in the network.}
	\label{HB}
\end{figure}
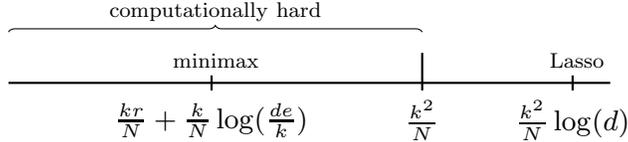

\noindent 
\textbf{Notation:} For any positive integer $n$, we denote by $[n]$ the set $\{1,\ldots,n\}$ and by $[[n]]$ the set of couples of $[n]$, of cardinality ${n \choose 2}$. We denote by $\R$ the set of real numbers and by $\fS^d$ the set of real symmetric matrices of size $d$. For a matrix $A\in \fS^d$, we denote by $\|A\|_F$ its Frobenius norm, defined by
\[
\|A\|_F^2 = \sum_{i,j \in [d]}A_{ij}^2\, .
\]
We extend this definition for \(B \in\fS^n \) and any subset $\Omega \subseteq [[n]]$ to its semi-norm $\|B\|_{F,\Omega}$ defined by
\[
\|B\|_{F,\Omega}^2 = \sum_{i,j\, : \, (i,j) \in \Omega}B_{ij}^2\, .
\]
The corresponding bilinear form playing the role of inner-product of two matrices \(B_1, B_2 \in \fS^n \) is denoted as  \(\llangle B_1, B_2\rrangle_{F,\Omega}\).  For a matrix $B\in \fS^n$, we also make use of the following matrix norms and pseudo-norms for \(p,q \in [0, \infty)\), with \(\|B\|_{p,q} = \big\| (\|B_{1*}\|_p  \cdots \|B_{d*}\|_p) \big\|_q\), where \(B_{i*}\) denotes the \(i\)th row of $B$, and \(\| B \|_{\infty} = \max_{(i,j) \in [d]}  |B_{ij}| \).
For asymptotic bounds,
we shall write  \(f(x) \lesssim g(x)\) if \(f(x)\) is bounded by a constant multiple of \(g(x)\).

\section{Problem description}
\label{SEC:desc}

\subsection{Generative model}
For a set of vertices $V = [n]$ and explanatory variables $X_i \in \R^d$ associated to each $i \in V$, a random graph $G=(V,E)$ is generated by the following model. For all $i,j \in V$, variables $X_i,X_j \in \R^d$ and an unknown matrix $\Theta_\star \in \fS_d$, an edge connects the two vertices $i$ and $j$ independently of the others according to the distribution
\begin{equation}
\label{eq:genModel}
\fP\big((i,j) \in E \big) = \sigma(X_i^\top \Theta_\star X_j) = \frac{1}{1+\exp(-X_i^\top \Theta_\star X_j)}\, .
\end{equation}
Here we denote by $\sigma$ the {\em sigmoid}, or {\em logistic} function. 
\begin{definition}
	\label{DefProbs}
	We denote by $\pi_{ij}: \fS_d \to [0,1]$ the function mapping a matrix $\Theta \in  \fS_d$ to the probability in \eqref{eq:genModel}. Let \(\Sigma \in \fS_n \) with \(\Sigma_{ij} = X_i^\top \Theta X_j \) denote the so-called affinity matrix. In particular, we then have \(\pi_{ij}(\Theta) = \sigma(\Sigma_{ij})\). 
\end{definition}
Our observation consists of the explanatory variables $X_i$ and of the observation of a subset of the graph. Formally, for a subset $\Omega \subseteq [[n]]$, we observe an adjacency vector $Y$ indexed by ${\Omega}$ that satisfies, for all $(i,j)\in \Omega$, $Y_{(i,j)} =1$ if and only if $(i,j) \in E$  (and 0 otherwise). We thus have 
\begin{equation}
\label{eq:genModel_2}
Y_{(i,j)} \sim \text{Bernoulli}\big(\pi_{ij}(\Ts)\big)\,, \quad (i,j)\in \Omega \,.
\end{equation}
The joint data distribution is denoted by \(\P_{\Ts}\)  and is thus completely specified by \(\pi_{ij}(\Ts)\), \((i,j)\in \Omega\).  For ease of notation, we write $N=|\Omega|$, representing the {\em effective sample size}.
Our objective is to estimate the parameter matrix $\Theta_\star$, based on the observations $Y \in \R^N$ and on known explanatory variables $\bX \in \R^{d \times n}$.

This problem can be reformulated as a classical logistic regression problem. Indeed, writing $\vect(A) \in \R^{d^2}$ for the {\em vectorized} form of a matrix $A\in \fS^d$, we have that
\begin{equation}\label{eq:vectorNot}
X_i^\top \Theta_\star X_j = \tr(X_j X_i^\top \Theta_\star) = \langle \vect(X_j X_i^\top), \vect(\Theta_\star) \rangle\, .
\end{equation}
The vector of observation $Y \in \R^{N}$ therefore follows a logistic distribution with explanatory design matrix $\bD_\Omega \in \R^{N \times d^2}$ such that $\bD_{\Omega \, (i,j)} = \vect(X_j X_i^\top)$ and predictor $\vect(\Theta_\star) \in \R^{d^2}$. We focus on the matrix formulation of this problem, and consider directly {\em matrix logistic regression} in order to simplify the notation of the explanatory variables and our model assumptions on $\Theta_\star$, that are specific to matrices. 

\subsection{Comparaison with other models}
This model can be compared to other settings in the statistical and learning literature.


{\bf Generalised linear model.}
As discussed above in the remark to  \eqref{eq:vectorNot}, this is an example of a logistic regression model. We focus in this work on the case where the matrix $\Theta_\star$ is block-sparse. 
The problem of sparse generalised linear models, and sparse logistic regression in particular has been extensively studied
\citep[see, e.g.][and references therein]{van2008high, bunea2008honest,meier2008group,bach2010self, rigollet2012kullback, abramovich2014}. Our work focuses on the more restricted case of block-sparse and low-rank parameter, establishing interesting statistical and computational phenomena in this setting.

{\bf Graphon model.} The graphon model is a model of a random graph in which the explanatory variables associated with the vertices in the graph are unknown. 
It has recently  become popular in the statistical community, see \cite{wolfe2013nonparametric, klopp2015oracle, gao2015rate, zhang2015estimating}. Typically, an objective of statistical inference is a link function which belongs  to either  
a parametric or non-parametric class of functions. Interestingly,  the minimax lower bound for the classes of 
H\"older-continuous functions, obtained 
in \cite{gao2015rate}, has not been attained by any polynomial-time algorithm.

{\bf Trace regression models.} The modelling assumption \eqref{eq:genModel} of the present paper is in fact very close to the trace regression model, as it follows from the representation  \eqref{eq:vectorNot}.
Thus, the block-sparsity and low-rank structures are preserved and can well be studied by the means of techniques developed 
for the trace  regression. 
We refer to \cite{koltchinskii2011, negahban2011estimation, rohde2011estimation, fan2016robust} for recent developments in the linear trace regression model, and 
\cite{fan2017} for the generalised trace regression model. However, computational lower bounds have not been studied either and many existing minimax optimal estimators cannot to be computed in polynomial time.

{\bf Metric learning.} In the task of metric learning, observations depend on an unknown geometric representation $V_1,\ldots,V_n$ of the variables in a Euclidean space of low dimension. The goal is to estimate this representation (up to a rigid transformation), based on noisy observations of $\langle V_i,V_j \rangle$ in the form of random evaluations of similarity. Formally, our framework also recovers the task of metric learning by taking $X_i = e_i$ and $\Theta_\star$  an unknown semidefinite positive matrix of small rank (here $V^\top V$), since
\[
\langle V_i,V_j \rangle = \langle Ve_i,Ve_j \rangle = e_i^\top \, V^\top  V \, e_j \, . 
\]
We refer to \citep{Chen09, Bellet14asurvey} and references therein for a comprehensive survey of metric learning methods.

\subsection{Parameter space}
\label{SEC:para}
The unknown predictor matrix $\Theta_\star$ describes the relationship between the observed features $X_i$ and the probabilities of connection $\pi_{ij}(\Theta_\star) = \sigma(X_i^\top \Theta_\star X_j)$ following Definition~\ref{DefProbs}. 
We focus on the high-dimensional setting where $d^2 \gg N$: the number of features for each vertex in the graph, and number of free parameters, is much greater than the total number of observations. In order to counter the curse of dimensionality, we make the assumption that the function $(X_i,X_j) \mapsto \pi_{ij}$ depends only on a small subset $S$ of size $k$ of all the coefficients of the explanatory variables. This translates to a {\em block-sparsity} assumption on $\Theta_\star$: the coefficients $\Theta_{\star \, ij}$ are only nonzero for $i$ and $j$ in $S$. 
Furthermore,
we assume that the rank of the matrix $\Ts$ can be smaller than the
size of the block. 
Formally, we define the following parameter spaces
\[
\cP_{k,r}(M) = \Big\{ \Theta \in \fS^d \, : \, \|\Theta\|_{1,1} < M \; \, , \, \|\Theta\|_{0,0} \le k \; ,\, \text{and} \; \rk(\Theta) \le r\Big\}\, ,
\]
for the coefficient-wise $\ell_1$ norm $\|\cdot\|_{1,1}$ on $\fS^d$
and integers $k,r \in [d]$. We also denote \(\cP(M) = \cP_{d,d}(M) \) for convenience.

\begin{remark}
	\label{RemarkParSpa}
	The bounds on block-sparsity and rank in our parameter space are structural bounds: we consider the case where the matrix $\Theta_\star$ can be concisely described in terms of the number of parameters. This is motivated by considering the spectral decomposition of the real symmetric matrix $\Ts$ as
	\[
	\Ts = \sum_{\ell =1}^r \lambda_\ell u_\ell u_\ell^\top\, .		
	\]
	The affinity $\Sigma_{ij} = X_i^\top \Ts X_j$ between vertices $i$ and $j$ is therefore only a function of the projections of $X_i$ and $X_j$ along the axes $u_\ell$, i.e.
	\[
	\Sigma_{ij} = X_i^\top \Ts X_j = \sum_{\ell =1}^r \lambda_\ell (u_\ell^\top X_i) (u_\ell^\top X_j)\, .
	\]	
	Assuming that there are only a few of these directions $u_\ell$ with non-zero impact on the affinity motivates the low-rank assumption, while assuming that there are only few relevant coefficients of $X_i, X_j$ that influence the affinity corresponds to a sparsity assumption on the $u_\ell$, or block sparsity of $\Ts$. The effect of these projections on the affinity is weighted by the $\lambda_\ell$. By allowing for negative eigenvalues, we allow our model to go beyond a geometric description, where close or similar $X$s are more likely to be connected. This can be used to model interactions where opposites attract. 		
\end{remark}

The assumption of simultaneously sparse and low-rank matrices arises naturally in many applications in statistics and machine learning and has attracted considerable recent attention.
Various regularisation techniques have been developed for estimation, variable and rank selection in multivariate regression problems \citep[see, e.g.][and the references therein]{bunea2012joint, richard2012estimation}.

\subsection{Explanatory variables}
\label{SEC:expl}
As mentioned above, this problem is different from tasks such as metric learning, where the objective is to estimate the $X_i$ with no side information. Here they are seen as covariates, allowing us to infer from the observation on the graph the predictor variable $\Theta_\star$. For this task to be even possible in a high-dimensional setting, we settle the identifiability issue by making the following variant of a classical assumption on $\bX \in \R^{d \times n}$.

\begin{definition}[Block isometry property]
	\label{AssWellCond}
	For a matrix $\bX \in \R^{d \times n}$ and an integer $s\in [d]$, we define $\Delta_{\Omega,s}(\bX) \in (0,1)$ as the smallest positive real such that
	\[
	N\big(1-\Delta_{\Omega,s}(\bX) \big) \|B\|_F^2 \le \|\bX^\top B\, \bX\|_{F,\Omega}^2 \le N\big(1+\Delta_{\Omega,s}(\bX)\big) \|B\|_{F}^2\, ,
	\]
	for all matrices $B \in \fS^d$ that satisfy the block-sparsity assumption $\|B\|_{0,0} \le s$.
\end{definition} 


\begin{definition}[Restriced isometry properties]
	For a matrix $A \in \R^{n \times p}$ and an integer $s\in [p]$, $\delta_{s}(A)\in (0,1)$ is the smallest positive real such that
	\[
	n\big(1-\delta_{s}(A) \big) \|v\|_2^2 \le \|Av\|_{2}^2 \le n \big(1+\delta_{s}(A) \big) \|v\|_2^2\, ,
	\]
	for all $s$-sparse vectors, i.e. satisfying $\|v\|_0\le s$.\\
	
	When $p=d^2$ is a square, we define $\delta_{\cB,s}(A)$ as the smallest positive real such that
	\[
	n\big(1-\delta_{\cB,s}(A) \big) \|v\|_2^2 \le \|Av\|_{2}^2 \le n\big(1+\delta_{\cB, s}(A) \big) \|v\|_2^2\, ,
	\]
	for all vectors such that $v=\vect(B)$, where $B$ satisfies the block-sparsity assumption $\|B\|_{0,0} \le s$.
\end{definition} 

The first definition is due to \cite{CanTao04}, with restriction to sparse vectors. It can be extended in general, as here, to other types of restrictions \citep[see, e.g.][]{TraGri15}. Since the restriction on the vectors in the second definition ($s$-by-$s$ block-sparsity) is more restricting than in the first one (sparsity),  $\delta_{\cB,s}$ is smaller than $\delta_{s^2}$. These different measures of restricted isometry are related, as shown in the following lemma
\begin{lemma}
	\label{lemma:RIP}
	For a matrix $\bX \in \R^{d \times n}$, let $\bD_\Omega \in \R^{N\times d^2}$ be defined row-wise by $\bD_{\Omega \,(i,j)} = \vect(X_j X_i^\top)$ for all $(i,j) \in \Omega$. It holds that 
	\[
	\Delta_{\Omega,s}(\bX) = \delta_{\cB,s}(\bD_\Omega)\, .
	\]
\end{lemma}

\begin{proof}
	This is a direct consequence of the definition of $\bD_\Omega$, which yields $\|\bX^\top B\, \bX\|_{F,\Omega}^2 = \|\bD_{\Omega} \, \vect(B)\|_2^2$, and  $\|\vect(B)\|_2^2 = \|B\|_F^2$.
	
\end{proof}

The assumptions above guarantee that the matrix $\Ts$ can be recovered from observations of the affinities, settling the well-posedness of this part of the inverse problem. However, we do not directly observe these affinities, but their image through the sigmoid function. We must therefore further 
impose the following assumption on the design matrix \(\bX\) that yields constraints on the probabilities \(\pi_{ij}\) and in essence governs the identifiability of \(\Theta_{\star}\).
\begin{assumption}
	\label{AssParamBound}
	There exists a constant \(M > 0\) such that for all \(\Theta \) in the class \(\cP(M)\) we have \(\max_{(i,j) \in \Omega} |X_i^\top \Theta X_j| < M \).
\end{assumption}
In particular, under this assumption a constant 
\begin{equation}
\label{ConstLm}
\Lm(M) := \sigma^\prime (M) = \sigma(M)\big(1 - \sigma(M)\big)\,,
\end{equation}
is lower bounded away from zero, and
we have 
\begin{equation}
\inf_{\Theta \in \cP(M)}  \sigma^\prime (X_i^\top \Theta X_j) \ge \Lm(M) > 0\,,
\label{infThGkM}
\end{equation}
for all \((i,j)\in \Omega \). Assuming that \(\Lm(M)\) always depends on the same \(M\), we  sometimes  write simply \(\Lm\).
\begin{remark}
	
	Assumption~\ref{AssParamBound} is necessary for the identifiability of \(\Theta_{\star}\): if $X_i^\top \Theta_\star X_j$ can be arbitrarily large in magnitude, $\pi_{ij} = \sigma(X_i^\top \Theta_\star X_j)$ can be arbitrarily close to $0$ or $1$. Since our observations only depend on $\Theta_\star$ through its image $\pi_{ij}$, this could lead to a very large estimation error on $\Theta_\star$ even with a small estimation error on the $\pi_{ij}$.

\end{remark}
\begin{remark}
	This assumption has already appeared in the literature on high-dimensional estimation, see \cite{van2008high, abramovich2014}.
	Similarly to \cite{bach2010self}, Assumption~\ref{AssParamBound} can be shown to be redundant  
	for minimax optimal prediction, 
	because the log-likelihood function in the matrix logistic regression model satisfies the so-called self-concordant property. 
	Our analysis to follow can be combined with an analysis similar to \cite{bach2010self} to get rid of the assumption for minimax optimal prediction.
\end{remark}

\begin{proposition}
	The identifiability assumption \(\max_{(i,j) \in \Omega} |X_i^\top \Theta X_j| < M \) 
	is guaranteed for  all \(\Theta \in \cP(M)\) and design matrices \(\bX\) satisfying either of the following
	\begin{itemize}
		\item  $  \|X_j X_i^\top\|_{\infty } \le 1$,
		\item $\|\Theta \|_{F}^2 < M_1$ for some \(M_1 > 0\) and the block isometry property. 
	\end{itemize}
\end{proposition}

\subsubsection{Random designs}
\label{SecRandDesign}
For random designs, we require the block isometry property to hold with high probability. 
Then the results in this article carry over directly
and thus we do not discuss it in full detail. 
It is well known that for sparse linear models with the dimension of a target 
vector \(\bar{p}\) and the sparsity \(\bar{k}\), 
the classical restricted isometry property holds for some classes of random matrices with i.i.d. entries including 
sub-Gaussian and Bernoulli matrices, see \cite{mendelson2008uniform}, provided that 
\(\bar{n} \gtrsim \bar{k} \log(\bar{p}/\bar{k})\), and i.i.d. subexponential random matrices, see \cite{adamczak2011restricted}, provided that 
\(\bar{n} \gtrsim \bar{k} \log^2(\bar{p}/\bar{k})\).
In the same spirit, the design matrices with independent entries following sub-Gaussian, subexponential or Bernoulli 
distributions can be shown to satisfy the block isometry property, cf.  \cite{WangBer2016}, provided that the number of observed edges in the network  satisfies \(N \gtrsim k^2 \log^2(d/k)\) for sub-Gaussian and subexponential designs and 
\(N \gtrsim k^2 \log(d/k)\) for Bernoulli designs.


\section{Matrix Logistic Regression}

\label{SEC:algo}

The log-likelihood for this problem is
\[
\ell_Y(\Theta) = -\sum_{(i,j) \in \Omega} \xi(s_{(i,j)} X_i^\top \Theta X_j)\, ,
\]
where $s_{(i,j)} = 2 Y_{(i,j)} -1$ is a sign variable that depends on the observations $Y$ and $\xi: x \mapsto \log(1+e^x)$ is a {\em softmax} function, convex on $\R$. As a consequence, the negative log-likelihood $ -\ell_Y$ is a convex function of $\Theta$. Denoting by $\ell$ the expectation $\E_{\Theta_\star}[\ell_Y]$, we recall the classical expressions for all $\Theta \in \fS^d$
\begin{align*}
\ell(\Theta) &= \ell(\Theta_\star) -  \sum_{(i,j) \in \Omega} {\sf KL}(\pi_{ij}(\Theta_\star),\pi_{ij}(\Theta))\\
&= \ell(\Theta_\star) - {\sf KL}(\fP_{\Theta_\star},\fP_{\Theta})\, ,
\end{align*}
where we recall $\pi_{ij}(\Theta) = \sigma(X_i^\top \Theta X_j)$, and
\[
\ell_Y(\Theta) = \ell(\Theta) + \llangle \nabla \zeta , \Theta \rrangle_F \,,
\]
where $\zeta$ is a stochastic component of the log-likelihood  with constant gradient 
$\nabla \zeta \in \R^{d\times d}$ given by  $\nabla \zeta= \sum_{(i,j)\in \Omega} (Y_{(i,j)} - \pi_{ij}(\Theta_\star)) X_j X_i ^\top$, which is a sum of independent centered random variables.

\subsection{Penalized logistic loss}
\label{SEC:penLogLoss}

In a classical setting where $d$ is fixed and $N$ grows, the maximiser of $\ell_Y$ - the maximum likelihood estimator   - is an accurate estimator of $\Theta_\star$, provided that it is possible to identify $\Theta$ from $\fP_{\Theta}$ (i.e. if the $X_i$ are well conditioned). We are here in a high-dimensional setting where $d^2 \gg N$, and this approach is not directly possible. Our parameter space indicates that the intrinsic dimension of our problem is truly much lower in terms of rank and block-sparsity. Our assumption on the conditioning of the $X_i$ is tailored to this structural assumption. In the same spirit, we also modify our estimator in order to promote the selection of elements of low rank and block-sparsity.
Following the ideas of 
\cite{birge2007minimal} and \cite{abramovich2014}, 
we define the following penalized maximum likelihood estimator 
\begin{equation}
\label{eq:penMLE}
\hat \Theta \in \argmin_{\Theta \in\cP(M)} \Big\{ - 	\ell_Y(\Theta)  + p(\Theta) \Big\}\, ,
\end{equation}

with a penalty $p$ defined as
\begin{equation}
\label{eq:pen}
p(\Theta)= g(\text{rank}(\Theta),\|\Theta\|_{0,0}) \, ,\quad \text{and} \quad g(R,K) =  c K R  + c K\log\big(\frac{d\ex}{K}\big) \, ,
\end{equation}
where \(c > 0\) is a universal constant and to be specified further. 
The proof of the following theorem is based on Dudley's integral argument  combined with Bousquet's inequality and is deferred to the Appendix. 
\begin{theorem}
	\label{ThmUpperBound}
	Assume  the design matrix \(\bX\) satisfies $\max_{(i,j) \in \Omega} |X_i^\top \Theta_\star X_j| < M$ for some \(M > 0\) and all \(\Ts\) in a given class, and the penalty term \(p(\Theta)\) satisfies \eqref{eq:pen} with the constants \(c \ge c_{1}/\Lm\), \(c_1 > 1\), \(\Lm\) given in \eqref{ConstLm}. 
	Then for the penalised MLE estimator \(\Th\), the following non-asymptotic upper bound on  the 
	expectation of the Kullback-Leibler divergence between the measures \(\P_{\Ts}\) and \(\P_{\Th}\) holds
	\begin{equation}
	\label{rateUpperBound0}
	\sup_{\Ts \in \cP_{k,r}(M)  } \frac{1}{N}
	\E[\Kc(\P_{\Ts}, \P_{\Th})] \le
	C_1\frac{kr}{N} +  C_1\frac{k}{N} \log\big(\frac{d\ex}{k}\big)  \,,
	\end{equation}
	where \(C_1 > 3c\) is some universal constant for all \(k = 1,...,d\) and \(r = 1,...,k\). 
\end{theorem}
\begin{remark}
	Random designs with i.i.d. entries following sub-Gaussian, Bernoulli and subexponential distributions 
	discussed in Section~\ref{SecRandDesign} yield the same rate as well. It can formally be shown using standard conditioning arguments \citep[see, e.g.][]{nickl2013}.
\end{remark}

\begin{corollary}
	Assume  the design matrix \(\bX\) satisfies  the block isometry property from Definition~\ref{AssWellCond} and $\max_{(i,j) \in \Omega} |X_i^\top \Theta_\star X_j| < M$ for some \(M > 0\) and all \(\Ts\) in a given class, and the penalty term \(p(\Theta)\) is as in Theorem~\ref{ThmUpperBound}. 
	Then for the penalised MLE estimator \(\Th\), the following non-asymptotic upper bound on the rate of estimation holds
	\begin{equation*}
	\sup_{\Ts \in \cP_{k,r}(M)  } \E \big[\|\Th - \Ts \|_F^2\big] \le 
	\frac{C_1}{\Lm(M) \big(1 - \Delta_{\Omega,2k}(\bX) \big)} \Big(	\frac{kr}{N} +  \frac{k}{N} \log\big(\frac{d\ex}{k}\big) \Big) \,,
	\end{equation*}
	where \(C_1 > 3c\) is some universal constant	for all \(k = 1,...,d\) and \(r = 1,...,k\). 
\end{corollary}

Let us define rank-constrained maximum likelihood estimators with bounded  block size as 
\begin{equation*}
\hat \Theta_{k,r} \in \argmin_{\Theta \in\cP_{k,r}(M)} \{  - \ell_Y(\Theta) \} \,.
\end{equation*}
It is intuitively clear that without imposing any regularisation on the likelihood function, the maximum likelihood 
approach selects the most complex model. In fact, the following result holds.
\begin{theorem}
	\label{ThmConstrainedMLE}
	
	Assume the  design matrix \(\bX\) satisfies the block isometry property from Definition~\ref{AssWellCond} and $\max_{(i,j) \in \Omega} |X_i^\top \Ts X_j| < M$ for some \(M > 0 \) and all \(\Ts\) in a given class.
	Then for the maximum likelihood estimator
	\( \hat \Theta_{k,r} \), 
	the following non-asymptotic upper bound on the 
	rate of estimation  holds
	
	\begin{equation*}
	\label{rateUpperBound2}
	\sup_{\Ts \in \cP_{k,r}(M)  }  \E \big[\|\Th_{k,r} - \Ts \|_F^2\big] \le
	\frac{C_3}{\Lm(M) \big(1 - \Delta_{\Omega,2k}(\bX) \big)} \Big(	\frac{kr}{N} +  \frac{k}{N} \log\big(\frac{d\ex}{k}\big) \Big) \,,
	\end{equation*}
	for all \(k = 1,...,d\) and \(r = 1,...,k\) and some  constant \(C_3 > 0\). 
\end{theorem}

\begin{remark}
	The penalty \eqref{eq:pen} belongs to the class of the so-called minimal penalties, cf. \cite{birge2007minimal}. In particular, 
	a naive MLE approach with \(p(\Theta) = 0\) in \eqref{eq:penMLE} yields a suboptimal estimator as it follows from Theorem~\ref{ThmConstrainedMLE}.
\end{remark}

\subsection{Convex relaxation}
\label{SEC:feedback}

In practice, computation of the estimator \eqref{eq:penMLE} is often infeasible. 
In essence, in order to compute it, we need to compare the likelihood functions over all possible subspaces \(\cP_{k,r}(M)\).
Sophisticated step-wise model selection procedures allow to reduce the number of analysed models. 
However, they are not feasible in a high-dimensional setting either. 
We here consider the following estimator 
\begin{equation}
\label{RGLasso}
\Th_{Lasso} =  \argmin_{\Theta \in \fS^d } \{ - \ell_Y(\Theta) + \lambda \|\Theta \|_{1,1}   \} \,,
\end{equation}
with \(\lambda > 0\) to be chosen further, which is equivalent to the logistic Lasso on \(\vect(\Theta)\).  Using standard arguments, cf. Example 1 in  \cite{van2008high}, combined with the block isometry property 
the following result immediately follows.
\begin{theorem}
	\label{GRLassoTheorem}
	
	Assume the design matrix \(\bX\) satisfies the block isometry property from Definition~\ref{AssWellCond} and $\max_{(i,j) \in \Omega} |X_i^\top \Ts X_j| < M$ for some \(M > 0 \) and all \(\Ts\) in a given class. 	
	Then for \(\lambda = C_4\sqrt{\log d}\),  where \(C_4 > 0\) is an appropriate universal constant, 
	the estimator \eqref{RGLasso} satisfies 
	\begin{equation}
	\label{rateUpperBound1}
	\sup_{\Ts \in \cP_{k,r}(M) } \E \big[\|\Th_{Lasso} - \Ts \|_F^2\big] \le
	\frac{C_5}{\Lm(M) \big(1 - \Delta_{\Omega,2k}(\bX) \big)}
	\frac{k^2}{N} \log d\,,
	\end{equation}
	for all \(k = 1,...,d\) and \(r = 1,...,k\) and some universal constant \(C_5 > 0\).
\end{theorem}

As one could expect the upper bound on the rate of estimation  of our feasible 
estimator is independent of the true rank \(r\). 
It is natural, when dealing with a low-rank and block-sparse objective matrix,
to combine the nuclear penalty with either the \((2,1)\)-norm penalty or the \((1,1)\)-norm penalty of a matrix,
cf. \cite{giraud2011low, koltchinskii2011, bunea2012joint, richard2012estimation}. 
In our setting, it can be easily shown that combining the \((1,1)\)-norm penalty and the nuclear penalty yields the same 
rate of estimation \((k^2 / N)\log d\).
This appears to be inevitable in view of  
a computational lower bound, obtained in Section~\ref{SEC:complow}, 
which  is independent of the rank as well. In particular, these findings partially answer a question posed in Section~6.4.4 in
\cite{giraud2014introduction}.

\subsection{Prediction}
\label{SEC:slow}



In applications, as new users join the network, we are interested in predicting 
the probabilities of the links between them and the existing users. 
It is natural to measure the 
prediction error of an estimator \(\Th\)  by \(	\E\big[ \sum_{(i,j)\in \Omega} (\pi_{ij}(\Th) - \pi_{ij}(\Ts))^2\big] \) which is controlled according to the following result using the smoothness of the logistic function \(\sigma\).	


\begin{theorem}
	\label{ThmPredUpperBound}
	Under Assumption~\ref{AssParamBound}, 
	we have the following rate for estimating the 
	matrix of probabilities \( \Sigs = \bX^\top \Ts \bX \in \R^{n\times n}\) with the estimator \(\Sigh = \bX^\top \Th \bX\in \R^{n\times n}\):
	\[
	\sup_{\Ts \in \cP_{k,r}(M)  } \frac{1}{2N}\E\big[ \|\Sigh - \Sigs \|_{F,\Omega}^2\big] \le
	\frac{C_1}{\Lm(M) } \Big( \frac{kr}{N} + \frac{k}{N} \log\big(\frac{d\ex}{k}\big) \Big) \,,
	\]
	with the constant \(C_1\) from \eqref{rateUpperBound0}.
	The rate is minimax optimal, i.e. a minimax lower bound of the same asymptotic order holds for the prediction error of estimating the 
	matrix of probabilities \( \Sigs = \bX^\top \Ts \bX \in \R^{n\times n}\).
\end{theorem}

\subsection{Information-theoretic lower bounds}
\label{SEC:inflow}
\label{MinimaxLowerBound}
The following result demonstrates 
that the minimax lower bound on the rate of estimation matches the upper bound in Theorem \ref{ThmUpperBound} 
implying that the rate of estimation is minimax optimal. 
\begin{theorem}
	\label{LowerBoundThm} 
	Let the  design matrix \(\bX\) satisfy the block isometry property. Then for estimating \(\Ts \in \cP_{k,r}(M) \) in the matrix logistic regression model, 
	the following lower bound on the rate of estimation holds
	\[
	\inf_{\Th} \sup_{\substack{\Ts \in  \cP_{k,r}(M)} } \E \big[\| \Th - \Ts\|^2_F\big] \ge 
	\frac{C_2}{(1 + \Delta_{\Omega,2k}(\bX))}\Big(\frac{kr}{N} + \frac{k}{N}\log\big(\frac{d\ex}{k}\big) \Big)\,,
	\]
	where the constant \(C_2 > 0\) is independent of \(d,k,r\) and the infimum extends over all estimators \(\Th\). 
\end{theorem}

\begin{remark}
	The lower bounds of the same order hold 
	for the   
	expectation of the Kullback-Leibler divergence between the measures \(\P_{\Ts}\) and \(\P_{\Th}\) and  the prediction error of estimating the 
	matrix of probabilities \( \Sigs = \bX^\top \Ts \bX \in \R^{n\times n}\).
\end{remark}


\section{Computational lower bounds}
\label{SEC:complow}

In this section, we investigate whether the lower bound in Theorem~\ref{LowerBoundThm} can be achieved with an estimator 
computable in polynomial time. The fastest rate of estimation  attained by a (randomised) polynomial-time algorithm 
in the worst-case scenario is usually referred to as a \emph{computational lower bound}.
Recently, the gap between computational and statistical lower bounds  
has attracted a lot of attention in the statistical community. We refer to  \cite{ berthet2013complexity,berthet2013optimal,wang2014statistical,gao2014sparse, zhang2017tensor, Hajek151, Chen2015a, ma2015computational, ChenXu2016} 
for computational lower bounds in high-dimensional statistics based on the planted clique problem (see below), \cite{BerEll15} using hardness of learning parity with noise \cite{oymak2015simultaneously} for 
denoising of sparse
and low-rank matrices, \cite{Agarwal2012} for computational trade-offs in statistical learning, as well as \cite{ZhaWaiJor14} for worst-case lower bounds for sparse estimators in linear regression, as well as \cite{BruTroCev15,ChaJor13} for another approach on computational trade-offs in statistical problems, as well as \cite{BerCha16,BerPer17} on the management of these trade-offs. 
In order to establish a computational lower bound for the block-sparse matrix logistic regression, we exploit a reduction 
scheme from \cite{berthet2013complexity}: we show that detecting a subspace of \(   \cP_{k,r}(M)\) 
can be computationally as hard as solving  
the dense subgraph detection problem.

\subsection{The dense subgraph detection problem}
Although our work is related to the study of graphs, we recall for absolute clarity the following 
notions from graph theory. 
A \emph{graph} \(G = (V,E )\) is a non-empty set \(V\) of \emph{vertices}, together with a set 
\(E\) of distinct unordered pairs \(\{i,j \}\) with \(i,j  \in V\), \(i \neq j\). Each element 
\(\{i,j \}\) of \(E\) is an edge and joins \(i\) to \(j\). The vertices of an edge are called its endpoints.
We consider only undirected graphs with neither loops nor multiple edges. 
A graph is called \emph{complete} if every pair of distinct vertices is connected.
A graph \(G^\prime = (V^\prime,E^\prime )\) is a subgraph of a graph \(G = (V,E ) \) 
if \(V^\prime \subseteq V\) and \(E^\prime \subseteq E\). A subgraph \(C\) is called a \emph{clique}
if it is complete. The problem of detecting a maximum clique, or all cliques, the so-called Clique problem, 
in a given graph is known to be \(\NP\)-complete, 
cf. \cite{karp1972reducibility}.

The Planted Clique problem, motivated as an average case version of the Clique problem,  
can be formalised as a decision problem over random graphs, 
parametrised by the number of vertices \(n\) and  the size of the subgraph  \(k\). 
Let \(\mathbb{G}_n\) denote the collection of all graphs with \(n\) vertices 
and \(G(n, 1/2)\) denote distribution of Erd\"os-R\'enyi random graphs, uniform on \(\mathbb{G}_n\), 
where each edge is drawn independently at random with probability \(1/2\).
For any
\(k \in \{ 1, . . . , n\}\) and \(q \in (1/2,1]\), let \(G(n,1/2,k,q)\) be a distribution on \(\mathbb{G}_n\) constructed by first
picking \(k\) vertices independently at random
and connecting all edges in-between  with probability \(q \),
and then joining
each remaining pair of distinct vertices by an edge independently at random with probability \(1/2\). 
Formally, the Planted Clique problem refers to the hypothesis testing problem of 
\begin{equation}
H_0: A \sim G(n,1/2) \quad \text{vs.} \quad H_1: A \sim G(n,1/2,k,1)\,,
\label{HypoTestPlCliq1}
\end{equation}
based on observing an adjacency matrix \(A \in \R^{n\times n}\) of a random graph drawn from either \(G(n,1/2)\)
or \(G(n,1/2,k,1)\).

One of the main  properties of the Erd\"os-R\'enyi random graph were studied in \cite{ErdRen59}, as well as in \cite{grimmett1975colouring}, who in particular proved that the size of the largest clique in \(G(n,1/2)\) is asymptotically close to \(2\log_2 n\) almost surely. 
On the other hand, \cite{alon1998} proposed a spectral method that for 
\(k > c \sqrt{n}\) detects a planted clique with high probability in polynomial time. Hence the most intriguing regime for \(k\) 
is 
\begin{equation}
2\log_2n \le  k \le c \sqrt{n}\,.
\label{HardnessRegimerr}
\end{equation}
The conjecture that no polynomial-time algorithm exists for distinguishing between hypotheses in \eqref{HypoTestPlCliq1} in the regime 	\eqref{HardnessRegimerr}  with probability tending to \(1\) as \(n \to \infty\)
is the famous Planted Clique conjecture in complexity theory.
Its variations have been used extensively as computational hardness assumptions in statistical problems, see \cite{berthet2013optimal,wang2014statistical, gao2014sparse, cai2018}. 

The Planted Clique problem can be reduced to the so-called dense subgraph detection problem of testing the null hypothesis in \eqref{HypoTestPlCliq1} against the alternative \(H_1: A \sim G(n,1/2,k,q)\), where \(q \in (1/2,1]\). This is clearly a computationally harder problem.
In this paper, we assume the following variation of the Planted Clique conjecture which is used to establish a computational lower bound in the matrix logistic regression model. 



\begin{conjecture}[The dense subgraph detection conjecture] 
	\label{Conjecture}
	For any sequence \(k = k_n\) such that \(k \le n^\beta\) for some \(0 < \beta < 1/2\), and any \(q \in (1/2,1]\), there is no 
	(randomised) polynomial-time algorithm that can correctly identify the dense subgraph 
	with probability tending to \(1\) as \(n \to \infty\), i.e. for any sequence of (randomised) polynomial-time tests 
	\((\psi_n: \mathbb{G}_n \to \{0,1 \})_n\), we have 
	\[
	\liminf_{n \to \infty}  \big\{ \P_0(\psi_n(A) = 1) + \P_1(\psi_ n(A) = 0)\big\} \ge 1/3\,.
	\]
\end{conjecture}

\subsection{Reduction to the dense subgraph detection problem and a computational lower bound}
\label{ReductionSect}
Consider the  vectors of explanatory variables    \(X_i = N^{1/4} e_i\), \(i =1,..., n\) and assume without loss of generality that 
the observed set of edges \(\Omega\) in the matrix logistic regression model consists of the interactions of the  
\(n \) nodes  \(X_i\), i.e. it holds \(N = |\Omega| ={n \choose 2} \).
It follows from the matrix logistic regression modelling assumption \eqref{eq:genModel}
that the Erd\"os-R\'enyi graph \(G(n,1/2)\) corresponds to a random graph associated with the 
matrix \(\Theta_0 = 0 \in \R^{d\times d}\). 
Let \( \G_{l}(k)\) be a subset of \(\cP_{k,1}(M)\) with a fixed support \(l\) of the block.
In addition, let 
\(\G_k^{\alpha_N} \subset \cP_{k,1}(M)\) be a subset consisting of the matrices \(\Theta_l \in \G_{l}(k), l=1,...,K\), \(K = {n \choose k}\) such that 
all elements in the block of a matrix \(\Theta_l\) equal some \( \alpha_N = \alpha/\sqrt{N} >0\), see  Figure~\ref{HB2}.
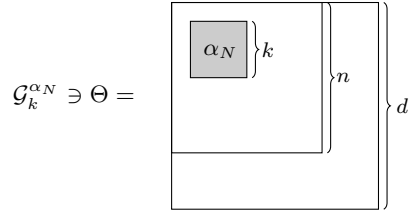
\begin{figure}[tp]
	\centering
	\begin{tikzpicture}[every node/.style={minimum size=.5cm-\pgflinewidth, outer sep=0pt}]
	\draw[fill=gray!40!white] (-0.75,0) rectangle (0,0.75);
	\draw (-1,-1) rectangle (1,1);
	\draw[decoration={brace,mirror},decorate]
	(0.07,0) -- node[xshift=6pt] {\text{\scriptsize \(k\)} } (0.07,0.75);
	\draw[decoration={brace,mirror},decorate]
	(1.07,-1) -- node[xshift=6pt] {\text{\scriptsize \(n\)} } (1.07,1);
	\draw[decoration={brace,mirror},decorate]
	(1.82,-1.75) -- node[xshift=7pt] {\text{\scriptsize \(d\)} } (1.82,1);
	\draw (-1,-1.75) rectangle (1.75,1);
	\node at (-0.35,0.35) {\(\alpha_N\)};
	
	\node at (-2.3,-0.25) {\(\G_k^{\alpha_N} \ni \Theta  = \)};
	\end{tikzpicture}
	\caption{The construction of matrices \(\G_k^{\alpha_N} \)  used in the reduction scheme.}
	\label{HB2}
\end{figure} 
Then we have 
\[
\fP\big((i,j) \in E   | X_i, X_j) = \frac{1}{1 + e^{-X_i^\top \Theta X_j}} = \frac{1}{1 + e^{-\alpha}}\,,
\]
for all \(\Theta \in \G_k^{\alpha_N}\).
Therefore, the testing problem 
\begin{equation}
H_0: Y \sim \P_{\Theta_0} \quad \text{vs.} \quad H_1: Y \sim \P_{\Theta}, \Theta \in \G_k^{\alpha_N}\,,
\label{HypoTest2}
\end{equation}
where  \(Y \in \{0,1\}^{N}\) is the adjacency vector of binary responses in the matrix logistic regression model, is reduced to
the dense subgraph detection problem with \(q =  1/ (1 + e^{-\alpha})\). 
This reduction scheme suggests that the computational lower bound for separating the hypotheses in the dense subgraph detection problem 
mimics the computational lower bound for separating the hypotheses in \eqref{HypoTest2} in the matrix logistic regression model.
The following theorem exploits this fact in order to establish a computational lower bound of order \(k^2 / N\) for estimating the matrix \(\Ts \in \cP_{k,r}(M)\).

\begin{theorem}
	Let \(\mathcal{F}_k\) be any class of matrices containing \(\G_k^{\alpha_N} \cup \Theta_0\) from the reduction scheme. Let 
	\(c > 0 \) be a positive constant and \(f(k,d,N)\) be a real-valued function 
	satisfying \(f(k,d,N) \le ck^2/N\) for \(k = k_{n} < n^{\beta}\), \(0 < \beta < 1/2\) and a sequence \(d = d_{n}\),
	for all \(n> m_0 \in \N\).
	If Conjecture~\ref{Conjecture} holds, for some the design \(\bX\) that fulfils  the block isometry property from Definition~\ref{AssWellCond}, 
	there is no estimator of  \(\Ts \in \mathcal{F}_k\), 
	that attains the rate \(f(k,d,N)\) for the  Frobenius norm risk,
	and can be evaluated using a (randomised)
	polynomial-time algorithm, i.e. for any estimator \(\Th\), computable in polynomial time, there exists a sequence \((k,d,N) = (k_{n}, d_{n},N)\),
	such that 
	\begin{equation}
	\label{CompLowerBound}
	\frac{1}{f(k,d,N)} \sup_{\Ts \in \mathcal{F}_k }  \E\big[\|\Th - \Ts \|_F^2\big] \to \infty \,,
	\end{equation}
	as \(n \to \infty\).
	Similarly, for any estimator \(\Th\), computable in polynomial time, 
	there exists a sequence \((k,d,N) = (k_{n}, d_{n},N)\), such that 
	\begin{equation}
	\label{CompLowerBoundPred}
	\frac{1}{f(k,d,N)} \sup_{\Ts \in \mathcal{F}_k} \frac{1}{N}\E\big[\|\Sigh - \Sigs\|_{F,\Omega}^2\big] \to \infty \,,
	\end{equation}
	for the prediction error of estimating \(\Sigs = \bX^\top \Ts \bX\).
\end{theorem}
\begin{remark}
	Thus the computational lower bound for estimating the matrix \(\Ts\) in the matrix logistic regression model is of order \(k^2/N\) compared to the minimax rate of estimation  of order \(kr/N + (k / N )\log(de/k)\) and the rate of estimation  \((k^2/ N) \log(d)\) for the Lasso estimator \(\Th_{Lasso}\), cf. Figure~\ref{HB}. Hence the computational gap is most noticeable for the matrices of rank \(1\).
	Furthermore, as a simple consequence of this result, the corresponding computational lower bound for the prediction risk of estimating 
	\(\Sigs = \bX^\top \Ts \bX\) is 
	\(k^2/N \) as well.
\end{remark}
\begin{proof} 
	We here provide a proof of the computational lower bound on the prediction error \eqref{CompLowerBoundPred} for convenience. 
	The bound on the estimation error 
	\eqref{CompLowerBound}
	is straightforward to show by 
	utilizing  the block isometry property. 
	Assume that there exists  a hypothetical  estimator \(\Th\) computable in polynomial time
	that attains the rate \(f(k,d,N)\) for the prediction 
	error, i.e. such that it holds that
	\[
	\limsup_{n \to \infty} \frac{1}{f(k,d,N)} \sup_{\Ts \in \mathcal{F}_k} 
	\frac{1}{N}\E\big[\|\bX^\top (\Th - \Ts) \bX\|_{F,\Omega}^2\big] \le b <  \infty \,,
	\]
	for all sequences \((k,d,N) = (k_{n}, d_{n},N)\) and a constant \(b\). Then by Markov's inequality, we have
	\begin{equation}
	\label{XTsThX}
	\frac{1}{N} \| \bX^\top (\Th - \Ts) \bX\|_{F,\Omega}^2 \le u f(k,d,N)\,,
	\end{equation}
	for some numeric constant \(u > 0\) with probability \(1 - b/u\)  for all \(\Ts \in \mathcal{F}_k\). 
	Following the reduction scheme, we consider the design vectors \(X_i = N^{1/4} e_i\), \(i = 1,...,n\) and 
	the subset of edges \(\Omega\), such that 
	\begin{equation}
	\label{DeeW}
	\frac{1}{N}	\| \bX^\top (\Th - \Ts) \bX\|_{F,\Omega}^2 = \sum_{(i,j) \in \Omega} ({\Th}_{ij} - {\Ts}_{ij})^2 = \|\Th - \Ts\|_{F,\Omega}^2\,,
	\end{equation}
	for any \(\Ts \in \G_k^{\alpha_N}\). Note, that the design vectors \(X_i = N^{1/4}e_i\), \(i = 1,...,n\) clearly 
	satisfy Assumption~\ref{AssParamBound}. 
	Thus, in order to separate the hypotheses 
	\begin{equation}	
	H_0: Y \sim \P_0 \quad \text{vs.} \quad H_1: Y \sim \P_\Theta, \Theta \in \G_k^{\alpha_N}\,,
	\label{HypoTest3}
	\end{equation}
	it is natural to employ the following test
	\begin{equation}
	\label{TestPsi}
	\psi(Y) = \Ind\big(\|\Th\|_{F,\Omega} \ge  \tau_{d,k}(u)\big)\,,
	\end{equation}
	where \(\tau_{d,k}^2(u) =  u f(k,d,N) \). The type I error of this test is controlled 
	automatically due to \eqref{XTsThX} and \eqref{DeeW}, \(\P_0(\psi = 1) \le b/u\). For the type II error, we obtain
	\begin{align*}
	\sup_{\Theta \in \G_k^{\alpha_N}} \P_{\Theta} (\psi = 0) & =
	\sup_{\Theta \in \G_k^{\alpha_N}} \P_{\Theta} \big(\|\Th\|_{F,\Omega} < \tau_{d,k}(u)\big)                                           \\
	& \le
	\sup_{\Theta \in \G_k^{\alpha_N}} \P_{\Theta} \big(\|\Th- \Theta\|_{F,\Omega}^2 > \| \Theta\|_{F,\Omega}^2 -   \tau_{d,k}^2(u)\big) \le b/u\,,
	\end{align*}
	provided that  
	\[
	k(k-1) \alpha_N^2/2  \ge 2 \tau_{d,k}^2(u) =  2 u f(k,d,N)\,,
	\]
	which is true in the regime \(k \le n^{\beta}\),  \( \beta < 1/2\), and 
	\(\alpha^2 \ge 4u / c\), (hence \(\alpha_N^2 \ge 4u / (cN)\) ) by the definition of the function \(f(k,d,N)\). 
	Putting the pieces together, we obtain
	\[
	\limsup_{n \to \infty}  \big\{ \P_0(\psi(Y) = 1) + \sup_{\Theta \in \G_k^{\alpha_N}} \P_{\Theta}(\psi(Y) = 0)\big\} \le 2b/u < 1/3\,,
	\]
	for \(u > 6b\). 
	Hence, the test \eqref{TestPsi} 
	separates the hypotheses \eqref{HypoTest3}. 
	This contradicts Conjecture~\ref{Conjecture} and implies \eqref{CompLowerBoundPred}.

\end{proof}

\section{Concluding remarks}
\label{SEC:ext}

Our results shed further light on the emerging topic of statistical and computational trade-offs in high-dimensional estimation.
The matrix logistic regression model is very natural to study the connection between statistical accuracy and computational efficiency 
as the model is based on the study of a generative model for random graphs. It is also an extension of lower bound {\em for all} statistical procedures to a model with covariates, the first of its kind.

Our fundings suggest that the block-sparsity is a limiting model selection criterion for polynomial-time estimation in the logistic regression model. That is, imposing further structure, like an additional low-rank constraint, and thus reducing the number of studied models yields an expected gain in the minimax rate, but that 
gain can never be achieved by a polynomial-time algorithm. In this setting, this implies that with a larger parameter space, while the statistical rates might be worse, they might be closer to those that are computationally achievable. 
As an illustration, both efficient and minimax optimal estimation is possible for estimating sparse vectors in the high-dimensional linear regression model, see, 
for example, SLOPE for achieving the exact minimax rate in \cite{bogdan2015slope, Bellec2018},  extending upon previous results on the Danzig selector and Lasso in \cite{bickel2009,candes2007}.

The logistic regression  is also a representative of a large class of generalised linear models. Furthermore, the proof of the minimax lower bound 
on the rate of estimation  in Theorem~\ref{LowerBoundThm} can be extended to all generalised linear models. The combinatorial estimator \eqref{eq:penMLE} can well be used to achieve the minimax rate.
The computational lower bound then becomes a delicate issue. A more sophisticated reduction scheme is needed to relate the dense subgraph detection 
problem to an appropriate testing problem for a generalised linear model. Approaching this question might require notions of noise discretisation and 
Le Cam equivalence studied in \cite{ma2015computational}.

An interesting question is whether it is possible to adopt polynomial-time algorithms available for detecting a dense subgraph for estimating 
the target matrix in the logistic regression model in all sparsity regimes. A common idea behind those algorithms is to search a dense subgraph over 
the vertices of a high degree and thus substantially reduce the number of compared models of subgraphs. 
The network we observe in the logistic regression model is generated by a sparse matrix. We may still observe a fully connected network which is generated by a 
small block in the target matrix. Therefore, it is not yet clear how to adopt algorithms for dense subgraph detection to submatrix detection. It remains an open question to establish whether these results can be extended to any design matrix, and all parameter regimes.

\bibliographystyle{ims}
\bibliography{ref}

\begin{thebibliography}{83}
\expandafter\ifx\csname natexlab\endcsname\relax\def\natexlab#1{#1}\fi
\expandafter\ifx\csname url\endcsname\relax
  \def\url#1{\texttt{#1}}\fi
\expandafter\ifx\csname urlprefix\endcsname\relax\def\urlprefix{URL }\fi

\bibitem[{{Abbe}(2017)}]{Abb17}
\textsc{{Abbe}, E.} (2017).
\newblock Community detection and stochastic block models: recent developments.
\newblock \textit{arXiv:1703.10146} .

\bibitem[{{Abbe} and Sandon(2015)}]{AbbSan15}
\textsc{{Abbe}, E.} and \textsc{Sandon, C.} (2015).
\newblock Detection in the stochastic block model with multiple clusters: proof
  of the achievability conjectures, acyclic {BP}, and the
  information-computation gap.
\newblock \textit{arXiv:1512.09080} .

\bibitem[{Abramovich and Grinshtein(2016)}]{abramovich2014}
\textsc{Abramovich, F.} and \textsc{Grinshtein, V.} (2016).
\newblock Model selection and minimax estimation in generalized linear models.
\newblock \textit{IEEE Transactions on Information Theory} \textbf{62}
  3721--3730.

\bibitem[{Adamczak et~al.(2011)Adamczak, Litvak, Pajor and
  Tomczak-Jaegermann}]{adamczak2011restricted}
\textsc{Adamczak, R.}, \textsc{Litvak, A.}, \textsc{Pajor, A.} and
  \textsc{Tomczak-Jaegermann, N.} (2011).
\newblock Restricted isometry property of matrices with independent columns and
  neighborly polytopes by random sampling.
\newblock \textit{Constructive Approximation} \textbf{34} 61--88.

\bibitem[{Agarwal(2012)}]{Agarwal2012}
\textsc{Agarwal, A.} (2012).
\newblock \textit{{Computational Trade-offs in Statistical Learning}}.
\newblock Theses, {University of California, Berkeley}.
\newblock
  \url{http://research.microsoft.com/en-us/um/people/alekha/thesismain.pdf}.

\bibitem[{Alon et~al.(1998)Alon, Krivelevich and Sudakov}]{alon1998}
\textsc{Alon, N.}, \textsc{Krivelevich, M.} and \textsc{Sudakov, B.} (1998).
\newblock Finding a large hidden clique in a random graph.
\newblock \textit{Random Structures and Algorithms} \textbf{13} 457--466.

\bibitem[{Alon and Sudakov(1998)}]{AloKriSud98}
\textsc{Alon, N. K.~M.} and \textsc{Sudakov, B.} (1998).
\newblock {Finding a large hidden clique in a random graph}.
\newblock In \textit{{Proceedings of the {E}ighth {I}nternational {C}onference
  ``{R}andom {S}tructures and {A}lgorithms'' ({P}oznan, 1997)}}, vol.~13.

\bibitem[{Bach(2010)}]{bach2010self}
\textsc{Bach, F.} (2010).
\newblock Self-concordant analysis for logistic regression.
\newblock \textit{Electronic Journal of Statistics} \textbf{4} 384--414.

\bibitem[{{Banks} et~al.(2016){Banks}, {Moore}, Neeman and
  Netrapalli}]{BanMoo16}
\textsc{{Banks}, J.}, \textsc{{Moore}, C.}, \textsc{Neeman, J.} and
  \textsc{Netrapalli, P.} (2016).
\newblock {Information-theoretic thresholds for community detection in sparse
  networks}.
\newblock \textit{arXiv:1601.02658} .

\bibitem[{Bellec et~al.(2018)Bellec, Lecu{\'e} and Tsybakov}]{Bellec2018}
\textsc{Bellec, P.}, \textsc{Lecu{\'e}, G.} and \textsc{Tsybakov, A.} (2018).
\newblock Slope meets lasso: improved oracle bounds and optimality.
\newblock \textit{Annals of Statistics (to appear)} ArXiv:1605.08651.

\bibitem[{Bellet et~al.(2014)Bellet, Habrard and Sebban}]{Bellet14asurvey}
\textsc{Bellet, A.}, \textsc{Habrard, A.} and \textsc{Sebban, M.} (2014).
\newblock A survey on metric learning for feature vectors and structured data.

\bibitem[{Berthet and Chandrasekaran(2016)}]{BerCha16}
\textsc{Berthet, Q.} and \textsc{Chandrasekaran, V.} (2016).
\newblock Resource allocation for statistical estimation.
\newblock \textit{Proceedings of the IEEE} \textbf{104} 115--125.

\bibitem[{Berthet and Ellenberg(2015)}]{BerEll15}
\textsc{Berthet, Q.} and \textsc{Ellenberg, J.} (2015).
\newblock Detection of planted solutions for flat satisfiability problems .

\bibitem[{Berthet and Perchet(2017)}]{BerPer17}
\textsc{Berthet, Q.} and \textsc{Perchet, V.} (2017).
\newblock Fast rates for bandit optimization with upper-confidence frank-wolfe.
\newblock \textit{NIPS 2017, to appear} .

\bibitem[{Berthet and Rigollet(2013{\natexlab{a}})}]{berthet2013complexity}
\textsc{Berthet, Q.} and \textsc{Rigollet, P.} (2013{\natexlab{a}}).
\newblock Complexity theoretic lower bounds for sparse principal component
  detection.
\newblock In \textit{Conference on Learning Theory}.

\bibitem[{Berthet and Rigollet(2013{\natexlab{b}})}]{berthet2013optimal}
\textsc{Berthet, Q.} and \textsc{Rigollet, P.} (2013{\natexlab{b}}).
\newblock Optimal detection of sparse principal components in high dimension.
\newblock \textit{The Annals of Statistics} \textbf{41} 1780--1815.

\bibitem[{Berthet et~al.(2016)Berthet, Rigollet and Srivastava}]{BerRigSri16}
\textsc{Berthet, Q.}, \textsc{Rigollet, P.} and \textsc{Srivastava, P.} (2016).
\newblock Exact recovery in the ising blockmodel .

\bibitem[{Biau and Bleakly(2008)}]{Biau2006}
\textsc{Biau, G.} and \textsc{Bleakly, K.} (2008).
\newblock Statistical inference on graphs.
\newblock \textit{Statistics and decisions} \textbf{24} 209--232.

\bibitem[{Bickel and Chen(2009)}]{Bickel21068}
\textsc{Bickel, P.~J.} and \textsc{Chen, A.} (2009).
\newblock A nonparametric view of network models and newman{\textendash}girvan
  and other modularities.
\newblock \textit{Proceedings of the National Academy of Sciences} \textbf{106}
  21068--21073.

\bibitem[{Bickel et~al.(2009)Bickel, Ritov and Tsybakov}]{bickel2009}
\textsc{Bickel, P.~J.}, \textsc{Ritov, Y.} and \textsc{Tsybakov, A.} (2009).
\newblock Simultaneous analysis of lasso and dantzig selector.
\newblock \textit{Ann. Statist.} \textbf{37} 1705--1732.
\newline\urlprefix\url{https://doi.org/10.1214/08-AOS620}

\bibitem[{Birg{\'e} and Massart(2007)}]{birge2007minimal}
\textsc{Birg{\'e}, L.} and \textsc{Massart, P.} (2007).
\newblock Minimal penalties for gaussian model selection.
\newblock \textit{Probability theory and related fields} \textbf{138} 33--73.

\bibitem[{Bogdan et~al.(2015)Bogdan, van~den Berg, Sabatti, Su and
  Cand{\`e}s}]{bogdan2015slope}
\textsc{Bogdan, M.}, \textsc{van~den Berg, E.}, \textsc{Sabatti, C.},
  \textsc{Su, W.} and \textsc{Cand{\`e}s, E.} (2015).
\newblock Slope---adaptive variable selection via convex optimization.
\newblock \textit{The Annals of Applied Atatistics} \textbf{9} 1103.

\bibitem[{Bousquet(2002)}]{bousquet2002bennett}
\textsc{Bousquet, O.} (2002).
\newblock A bennett concentration inequality and its application to suprema of
  empirical processes.
\newblock \textit{Comptes Rendus Mathematique} \textbf{334} 495--500.

\bibitem[{Bruer et~al.(2015)Bruer, Tropp, Cevher and Becker}]{BruTroCev15}
\textsc{Bruer, J.~J.}, \textsc{Tropp, J.~A.}, \textsc{Cevher, V.} and
  \textsc{Becker, S.~R.} (2015).
\newblock Designing statistical estimators that balance sample size, risk, and
  computational cost.
\newblock \textit{IEEE Journal of Selected Topics in Signal Processing}
  \textbf{9} 612--624.

\bibitem[{Bubeck et~al.(2014)Bubeck, Ding, Eldan and R{\'a}cz}]{BubDinEld14}
\textsc{Bubeck, S.}, \textsc{Ding, J.}, \textsc{Eldan, R.} and
  \textsc{R{\'a}cz, M.} (2014).
\newblock Testing for high-dimensional geometry in random graphs .
\newline\urlprefix\url{https://arxiv.org/abs/1411.5713}

\bibitem[{Bunea(2008)}]{bunea2008honest}
\textsc{Bunea, F.} (2008).
\newblock Honest variable selection in linear and logistic regression models
  via $\ell_1$ and $\ell_1 + \ell_2$ penalization.
\newblock \textit{Electronic Journal of Statistics} \textbf{2} 1153--1194.

\bibitem[{Bunea et~al.(2012)Bunea, She and Wegkamp}]{bunea2012joint}
\textsc{Bunea, F.}, \textsc{She, Y.} and \textsc{Wegkamp, M.} (2012).
\newblock Joint variable and rank selection for parsimonious estimation of
  high-dimensional matrices.
\newblock \textit{The Annals of Statistics} \textbf{40} 2359--2388.

\bibitem[{Cai and Wu(2018)}]{cai2018}
\textsc{Cai, T.} and \textsc{Wu, Y.} (2018).
\newblock Statistical and computational limits for sparse matrix detection.
\newblock \textit{arXiv preprint arXiv:1801.00518} .

\bibitem[{Candes and Plan(2011)}]{candes2011tight}
\textsc{Candes, E.} and \textsc{Plan, Y.} (2011).
\newblock Tight oracle inequalities for low-rank matrix recovery from a minimal
  number of noisy random measurements.
\newblock \textit{IEEE Transactions on Information Theory} \textbf{57}
  2342--2359.

\bibitem[{Candes and Tao(2007)}]{candes2007}
\textsc{Candes, E.} and \textsc{Tao, T.} (2007).
\newblock The dantzig selector: Statistical estimation when p is much larger
  than n.
\newblock \textit{Ann. Statist.} \textbf{35} 2313--2351.
\newline\urlprefix\url{https://doi.org/10.1214/009053606000001523}

\bibitem[{Candes and Tao(2005)}]{CanTao04}
\textsc{Candes, E.~J.} and \textsc{Tao, T.} (2005).
\newblock {Decoding by Linear Programming}.
\newblock \textit{IEEE Trans. Information Theory} \textbf{51} 4203--4215.

\bibitem[{Chandrasekaran and Jordan(2013)}]{ChaJor13}
\textsc{Chandrasekaran, V.} and \textsc{Jordan, M.~I.} (2013).
\newblock Computational and statistical tradeoffs via convex relaxation.
\newblock \textit{Proceedings of the National Academy of Sciences} .

\bibitem[{Chen(2015)}]{Chen2015a}
\textsc{Chen, Y.} (2015).
\newblock Incoherence-optimal matrix completion.
\newblock \textit{IEEE Transactions on Information Theory} \textbf{61}
  2909--2923.

\bibitem[{Chen et~al.(2009)Chen, Garcia, Gupta, Rahimi and Cazzanti}]{Chen09}
\textsc{Chen, Y.}, \textsc{Garcia, E.~K.}, \textsc{Gupta, M.~R.},
  \textsc{Rahimi, A.} and \textsc{Cazzanti, L.} (2009).
\newblock Similarity-based classification: Concepts and algorithms.
\newblock \textit{J. Mach. Learn. Res.} \textbf{10} 747--776.

\bibitem[{Chen and Xu(2016)}]{ChenXu2016}
\textsc{Chen, Y.} and \textsc{Xu, J.} (2016).
\newblock Statistical-computational tradeoffs in planted problems and submatrix
  localization with a growing number of clusters and submatrices.
\newblock \textit{Journal of Machine Learning Research} \textbf{17} 1--57.
\newline\urlprefix\url{http://jmlr.org/papers/v17/14-330.html}

\bibitem[{Decelle et~al.(2011)Decelle, Krzakala, Moore and
  Zdeborov{\'a}}]{DecKrzMoo11}
\textsc{Decelle, A.}, \textsc{Krzakala, F.}, \textsc{Moore, C.} and
  \textsc{Zdeborov{\'a}, L.} (2011).
\newblock Asymptotic analysis of the stochastic block model for modular
  networks and its algorithmic applications.
\newblock \textit{Physical Review E} \textbf{84} 066106.

\bibitem[{Devroye et~al.(2011)Devroye, Gy{{\"o}}rgy, Lugosi and
  Udina}]{DevGyoLug11}
\textsc{Devroye, L.}, \textsc{Gy{{\"o}}rgy, A.}, \textsc{Lugosi, G.} and
  \textsc{Udina, F.} (2011).
\newblock High-dimensional random geometric graphs and their clique number.
\newblock \textit{Electronic Communications in Probability} \textbf{16}
  2481--2508.

\bibitem[{Dudley(1967)}]{DUDLEY1967290}
\textsc{Dudley, R.} (1967).
\newblock The sizes of compact subsets of hilbert space and continuity of
  gaussian processes.
\newblock \textit{Journal of Functional Analysis} \textbf{1} 290 -- 330.

\bibitem[{Erd{\H o}s and R{\'e}nyi(1959)}]{ErdRen59}
\textsc{Erd{\H o}s, P.} and \textsc{R{\'e}nyi, A.} (1959).
\newblock On random graphs.
\newblock \textit{Publicationes Mathematicae} \textbf{6} 290--297.

\bibitem[{Fan et~al.(2017)Fan, Gong and Zhu}]{fan2017}
\textsc{Fan, J.}, \textsc{Gong, W.} and \textsc{Zhu, Z.} (2017).
\newblock Generalized high-dimensional trace regression via nuclear norm
  regularization.
\newblock \textit{arXiv preprint arXiv:1710.08083} .

\bibitem[{Fan et~al.(2016)Fan, Wang and Zhu}]{fan2016robust}
\textsc{Fan, J.}, \textsc{Wang, W.} and \textsc{Zhu, Z.} (2016).
\newblock Robust low-rank matrix recovery.
\newblock \textit{arXiv preprint arXiv:1603.08315} .

\bibitem[{Gao et~al.(2015)Gao, Lu and Zhou}]{gao2015rate}
\textsc{Gao, C.}, \textsc{Lu, Y.} and \textsc{Zhou, H.} (2015).
\newblock Rate-optimal graphon estimation.
\newblock \textit{The Annals of Statistics} \textbf{43} 2624--2652.

\bibitem[{Gao et~al.(2017)Gao, Ma and Zhou}]{gao2014sparse}
\textsc{Gao, C.}, \textsc{Ma, Z.} and \textsc{Zhou, H.} (2017).
\newblock Sparse cca: Adaptive estimation and computational barriers.
\newblock \textit{arXiv preprint arXiv:1409.8565} .

\bibitem[{Gin{\'e} and Nickl(2016)}]{gine2015mathematical}
\textsc{Gin{\'e}, E.} and \textsc{Nickl, R.} (2016).
\newblock \textit{Mathematical Foundations of Infinite-Dimensional Statistical
  Models}.
\newblock Cambridge Series in Statistical and Probabilistic Mathematics,
  Cambridge University Press.

\bibitem[{Giraud(2011)}]{giraud2011low}
\textsc{Giraud, C.} (2011).
\newblock Low rank multivariate regression.
\newblock \textit{Electronic Journal of Statistics} \textbf{5} 775--799.

\bibitem[{Giraud(2014)}]{giraud2014introduction}
\textsc{Giraud, C.} (2014).
\newblock \textit{Introduction to high-dimensional statistics}.
\newblock Chapman \& Hall/CRC Monographs on Statistics \& Applied Probability,
  Taylor \& Francis.

\bibitem[{Grimmett and McDiarmid(1975)}]{grimmett1975colouring}
\textsc{Grimmett, G.} and \textsc{McDiarmid, C.} (1975).
\newblock On colouring random graphs.
\newblock In \textit{Mathematical Proceedings of the Cambridge Philosophical
  Society}, vol.~77. Cambridge Univ Press.

\bibitem[{Hajek et~al.(2015)Hajek, Wu and Xu}]{Hajek151}
\textsc{Hajek, B.}, \textsc{Wu, Y.} and \textsc{Xu, J.} (2015).
\newblock Computational lower bounds for community detection on random graphs.
\newblock In \textit{Proceedings of The 28th Conference on Learning Theory},
  vol.~40 of \textit{Proceedings of Machine Learning Research}.

\bibitem[{Hoff et~al.(2002)Hoff, Raftery and Handcock}]{Hoff2002}
\textsc{Hoff, P.~D.}, \textsc{Raftery, A.~E.} and \textsc{Handcock, M.~S.}
  (2002).
\newblock Latent space approaches to social network analysis.
\newblock \textit{Journal of the American Statistical Association} \textbf{97}
  1090--1098.

\bibitem[{Holland et~al.(1983)Holland, Laskey and Leinhardt}]{HolLasLei83}
\textsc{Holland, P.~W.}, \textsc{Laskey, K.~B.} and \textsc{Leinhardt, S.}
  (1983).
\newblock Stochastic blockmodels: First steps.
\newblock \textit{Social Networks} \textbf{5} 109 -- 137.

\bibitem[{Holland and Leinhardt(1981)}]{Holland81}
\textsc{Holland, P.~W.} and \textsc{Leinhardt, S.} (1981).
\newblock An exponential family of probability distributions for directed
  graphs.
\newblock \textit{Journal of the American Statistical Association} \textbf{76}
  33--50.

\bibitem[{Jain et~al.(2016)Jain, Jamieson and Nowak}]{JaiJamNow16}
\textsc{Jain, L.}, \textsc{Jamieson, K.} and \textsc{Nowak, R.} (2016).
\newblock Finite sample prediction and recovery bounds for ordinal embedding .

\bibitem[{Karp(1972)}]{karp1972reducibility}
\textsc{Karp, R.} (1972).
\newblock \textit{Reducibility among combinatorial problems}.
\newblock Springer.

\bibitem[{Klopp and Tsybakov(2015)}]{klopp2015oracle}
\textsc{Klopp, O.} and \textsc{Tsybakov, a. V.~N., A.} (2015).
\newblock Oracle inequalities for network models and sparse graphon estimation.
\newblock \textit{Annals of Statistics (to appear)} ArXiv:1507.04118.

\bibitem[{Koltchinskii et~al.(2011)Koltchinskii, Lounici and
  Tsybakov}]{koltchinskii2011}
\textsc{Koltchinskii, V.}, \textsc{Lounici, K.} and \textsc{Tsybakov, A.}
  (2011).
\newblock Nuclear-norm penalization and optimal rates for noisy low-rank matrix
  completion.
\newblock \textit{The Annals of Statistics} \textbf{39} 2302--2329.

\bibitem[{Ku\v{c}era(1995)}]{Kuc95}
\textsc{Ku\v{c}era, L.} (1995).
\newblock {Expected complexity of graph partitioning problems}.
\newblock \textit{Discrete Appl. Math.} \textbf{57} 193--212.
\newblock Combinatorial optimization 1992 (CO92) (Oxford).

\bibitem[{Ma and Wu(2015)}]{ma2015computational}
\textsc{Ma, Z.} and \textsc{Wu, Y.} (2015).
\newblock Computational barriers in minimax submatrix detection.
\newblock \textit{The Annals of Statistics} \textbf{43} 1089--1116.

\bibitem[{Madeira and Oliveira(2004)}]{Madeira2004}
\textsc{Madeira, S.~C.} and \textsc{Oliveira, A.~L.} (2004).
\newblock Biclustering algorithms for biological data analysis: A survey.
\newblock \textit{IEEE/ACM Trans. Comput. Biol. Bioinformatics} \textbf{1}
  24--45.

\bibitem[{Massart(2007)}]{massart2007concentration}
\textsc{Massart, P.} (2007).
\newblock \textit{Concentration inequalities and model selection}, vol.~6.
\newblock Springer.

\bibitem[{Massouli{\'e}(2014)}]{Mas14}
\textsc{Massouli{\'e}, L.} (2014).
\newblock Community detection thresholds and the weak {Ramanujan} property.
\newblock In \textit{Proceedings of the 46th Annual ACM Symposium on Theory of
  Computing}. ACM.

\bibitem[{Meier et~al.(2008)Meier, van~de Geer and
  B{\"u}hlmann}]{meier2008group}
\textsc{Meier, L.}, \textsc{van~de Geer, S.} and \textsc{B{\"u}hlmann, P.}
  (2008).
\newblock The group lasso for logistic regression.
\newblock \textit{Journal of the Royal Statistical Society: Series B
  (Statistical Methodology)} \textbf{70} 53--71.

\bibitem[{Mendelson et~al.(2008)Mendelson, Pajor and
  Tomczak-Jaegermann}]{mendelson2008uniform}
\textsc{Mendelson, S.}, \textsc{Pajor, A.} and \textsc{Tomczak-Jaegermann, N.}
  (2008).
\newblock Uniform uncertainty principle for bernoulli and subgaussian
  ensembles.
\newblock \textit{Constructive Approximation} \textbf{28} 277--289.

\bibitem[{Mossel et~al.(2013)Mossel, Neeman and Sly}]{MosNeeSly13}
\textsc{Mossel, E.}, \textsc{Neeman, J.} and \textsc{Sly, A.} (2013).
\newblock A proof of the block model threshold conjecture.
\newblock \textit{arXiv:1311.4115} .

\bibitem[{Mossel et~al.(2015)Mossel, Neeman and Sly}]{MosNeeSly15}
\textsc{Mossel, E.}, \textsc{Neeman, J.} and \textsc{Sly, A.} (2015).
\newblock Reconstruction and estimation in the planted partition model.
\newblock \textit{Probability Theory and Related Fields} \textbf{162} 431--461.

\bibitem[{Negahban and Wainwright(2011)}]{negahban2011estimation}
\textsc{Negahban, S.} and \textsc{Wainwright, M.} (2011).
\newblock Estimation of (near) low-rank matrices with noise and
  high-dimensional scaling.
\newblock \textit{The Annals of Statistics} \textbf{39} 1069--1097.

\bibitem[{Nickl and van~de Geer(2013)}]{nickl2013}
\textsc{Nickl, R.} and \textsc{van~de Geer, S.} (2013).
\newblock Confidence sets in sparse regression.
\newblock \textit{Ann. Statist.} \textbf{41} 2852--2876.

\bibitem[{Oymak et~al.(2015)Oymak, Jalali, Fazel, Eldar and
  Hassibi}]{oymak2015simultaneously}
\textsc{Oymak, S.}, \textsc{Jalali, A.}, \textsc{Fazel, M.}, \textsc{Eldar, Y.}
  and \textsc{Hassibi, B.} (2015).
\newblock Simultaneously structured models with application to sparse and
  low-rank matrices.
\newblock \textit{IEEE Transactions on Information Theory} \textbf{61}
  2886--2908.

\bibitem[{Papa et~al.(2016)Papa, Bellet and Cl\'{e}men\c{c}on}]{PapaNIPS2016}
\textsc{Papa, G.}, \textsc{Bellet, A.} and \textsc{Cl\'{e}men\c{c}on, S.}
  (2016).
\newblock On graph reconstruction via empirical risk minimization: Fast
  learning rates and scalability.
\newblock In \textit{Advances in Neural Information Processing Systems 29}
  (D.~D. Lee, M.~Sugiyama, U.~V. Luxburg, I.~Guyon and R.~Garnett, eds.).
  Curran Associates, Inc., 694--702.

\bibitem[{Penrose(2003)}]{Pen03}
\textsc{Penrose, M.} (2003).
\newblock \textit{Random Geometric Graphs}.
\newblock Oxford University Press.

\bibitem[{Richard et~al.(2012)Richard, Savalle and
  Vayatis}]{richard2012estimation}
\textsc{Richard, E.}, \textsc{Savalle, P.} and \textsc{Vayatis, N.} (2012).
\newblock Estimation of simultaneously sparse and low-rank matrices.
\newblock In \textit{Proceedings of the 29th International Conference on
  Machine Learning (ICML-12)}.

\bibitem[{Rigollet(2012)}]{rigollet2012kullback}
\textsc{Rigollet, P.} (2012).
\newblock Kullback-leibler aggregation and misspecified generalized linear
  models.
\newblock \textit{The Annals of Statistics} \textbf{40} 639--665.

\bibitem[{Rohde and Tsybakov(2011)}]{rohde2011estimation}
\textsc{Rohde, A.} and \textsc{Tsybakov, A.} (2011).
\newblock Estimation of high-dimensional low-rank matrices.
\newblock \textit{The Annals of Statistics} \textbf{39} 887--930.

\bibitem[{Traonmilin and Gribonval(2015)}]{TraGri15}
\textsc{Traonmilin, Y.} and \textsc{Gribonval, R.} (2015).
\newblock Stable recovery of low-dimensional cones in hilbert spaces: One rip
  to rule them all .
\newline\urlprefix\url{https://arxiv.org/abs/1510.00504}

\bibitem[{Tsybakov(2008)}]{Tsybakov2008}
\textsc{Tsybakov, A.} (2008).
\newblock \textit{Introduction to Nonparametric Estimation}.
\newblock 1st ed. Springer.

\bibitem[{van~de Geer(2008)}]{van2008high}
\textsc{van~de Geer, S.} (2008).
\newblock High-dimensional generalized linear models and the lasso.
\newblock \textit{The Annals of Statistics} \textbf{36} 614--645.

\bibitem[{Wang et~al.(2016{\natexlab{a}})Wang, Berthet and Plan}]{WangBer2016}
\textsc{Wang, T.}, \textsc{Berthet, Q.} and \textsc{Plan, Y.}
  (2016{\natexlab{a}}).
\newblock Average-case hardness of rip certification.
\newblock In \textit{Proceedings of the 30th International Conference on Neural
  Information Processing Systems}. NIPS'16, Curran Associates Inc., USA.
\newline\urlprefix\url{http://dl.acm.org/citation.cfm?id=3157382.3157525}

\bibitem[{Wang et~al.(2016{\natexlab{b}})Wang, Berthet and
  Samworth}]{wang2014statistical}
\textsc{Wang, T.}, \textsc{Berthet, Q.} and \textsc{Samworth, R.}
  (2016{\natexlab{b}}).
\newblock Statistical and computational trade-offs in estimation of sparse
  principal components.
\newblock \textit{The Annals of Statistics} \textbf{44} 1896--1930.

\bibitem[{Wasserman and Faust(1994)}]{wasserman_faust_1994}
\textsc{Wasserman, S.} and \textsc{Faust, K.} (1994).
\newblock \textit{Social Network Analysis: Methods and Applications}.
\newblock Structural Analysis in the Social Sciences, Cambridge University
  Press.

\bibitem[{Wolfe and Olhede(2013)}]{wolfe2013nonparametric}
\textsc{Wolfe, P.} and \textsc{Olhede, S.} (2013).
\newblock Nonparametric graphon estimation.
\newblock \textit{arXiv preprint arXiv:1309.5936} .

\bibitem[{Yu et~al.(2008)Yu, Braun and Yildirim}]{Yu2008}
\textsc{Yu, H.}, \textsc{Braun, P.} and \textsc{Yildirim, M.} (2008).
\newblock High quality binary protein interaction map of the yeast interactome
  network.
\newblock \textit{Science (New York, NY)} \textbf{322} 104--110.

\bibitem[{Zhang and Dong(2017)}]{zhang2017tensor}
\textsc{Zhang, A.} and \textsc{Dong, X.} (2017).
\newblock Tensor svd: Statistical and computational limits.
\newblock \textit{arXiv preprint arXiv:1703.02724} .

\bibitem[{Zhang et~al.(2015)Zhang, Levina and Zhu}]{zhang2015estimating}
\textsc{Zhang, Y.}, \textsc{Levina, E.} and \textsc{Zhu, J.} (2015).
\newblock Estimating network edge probabilities by neighborhood smoothing.
\newblock \textit{arXiv preprint arXiv:1509.08588} .

\bibitem[{Zhang et~al.(2014)Zhang, Wainwright and Jordan}]{ZhaWaiJor14}
\textsc{Zhang, Y.}, \textsc{Wainwright, M.} and \textsc{Jordan, M.} (2014).
\newblock Lower bounds on the performance of polynomial-time algorithms for
  sparse linear regression \textbf{35}.

\end{thebibliography}

\appendix

\newpage
	\section{Proofs}
\label{SEC:proofs}

\subsection{Some geometric properties of the likelihood}
Let us recall  the \emph{stochastic component} of the likelihood function
\[
\label{StochComp}
\zeta(\Theta) = \Lc(\Theta) - \ell(\Theta)  = 
\sum_{(i,j) \in \Omega} \big(Y_{(i,j)} - \pi_{ij}(\Ts) \big)X_i^\top \Theta X_j \,,
\]
which is a linear function in \(\Theta\). 
The deviation of the gradient \(\nabla \zeta\) of the stochastic component is governed by the deviation of the independent 
Bernoulli random variables \(\varepsilon_{i,j} = Y_{(i,j)} - \E[Y_{(i,j)}] =Y_{(i,j)} - \pi_{ij}(\Ts) \), 
\((i,j) \in \Omega\). Let us introduce an upper triangular matrix  \(\mathcal{E}_\Omega = (\eps_{i,j})_{(i,j) \in \Omega}\) with zeros on the complement set \(\Omega^c\).
In this notation, we have \(\zeta(\Theta) = \llangle\zeta, \Theta \rrangle_F\),
with 
\[
\label{GradStochComp}
\nabla \zeta = \sum_{(i,j) \in \Omega} \varepsilon_{i,j} X_j X_i^\top = \bX \, \mathcal{E}_\Omega  \bX^\top \in \R^{d\times d}\,.
\]
In particular, \(\nabla \zeta\) is sub-Gaussian with parameter 
\( \sum_{(i,j)\in \Omega}\|X_j X_i^\top \|_F^2 / 4  = \| \bX^\top \bX\|_{F, \Omega}^2/4 \), i.e. it holds for the moment generating function of 
\(\llangle\zeta, B \rrangle_F\) for any \(B \in \R^{d\times d}\) and \(\sigma^2 = 1/4\),
\begin{align}
\varphi_{\llangle\zeta, B \rrangle_F }(t) & := \E\big[\exp({t \llangle\zeta, B \rrangle_F})\big] 
= \prod_{(i,j) \in \Omega} \E\big[\exp({t \varepsilon_{i,j} \llangle X_j X_i^\top, B \rrangle_F})\big] \nonumber \\
&\le \prod_{(i,j) \in \Omega} \exp\big({t^2 \sigma^2 \llangle X_j X_i^\top, B \rrangle_F}^2/ 2\big)
=   \exp\big({t \sigma^2  \| \bX^\top B \bX \|_{F, \Omega}^2} /2\big)  \,.
\label{varphillan}
\end{align}
We shall be frequently using versions of the following inequality, which is based on the fact that \(\nabla\ell(\Ts) = 0 \in \R^{d\times d}\), the
Taylor expansion and \eqref{infThGkM}, and holds for 
any \(\Theta \in \cP(M)\),
\begin{align}
\ell(\Ts) - \ell(\Theta)
&=  \frac{1}{2} \sum_{(i,j) \in \Omega}\big( \sigma^\prime (X_i^\top \Theta_0 X_j)\llangle X_j X_i^\top, \Ts - \Theta  \rrangle^2 \big) \nonumber
\\
&\ge
\frac{\Lm}{2} \sum_{(i,j) \in \Omega} \llangle X_j X_i^\top, \Ts - \Theta  \rrangle_F^2 
= \frac{\Lm}{2} \|\bX^\top( \Ts - \Theta ) \bX  \|_{F, \Omega}^2 \,,
\label{TayMean}
\end{align}
where \(\Theta_0 \in [\Theta, \Ts]\) element-wise. Furthermore, using that \(\sup_{t\in \R}\sigma^\prime(t) \le 1/4 \), we obtain for all 
\(\Theta \in \cP(M) \)
\[
\ell(\Ts) - \ell(\Theta)
\le
\frac{1}{8} \|\bX^\top( \Ts - \Theta ) \bX \|_{F,\Omega}^2 \,.
\label{TayUpperBound}
\]
We shall also be using the bounds
\begin{align}
\label{EmpProcBound}
\max_{(i,j) \in \Omega} \big(\varepsilon_{i,j}  X_i^T  (\Theta- \Ts) X_j \big)&\le \| \bX^\top  (\Theta- \Ts)\bX  \|_{F,\Omega}\,, \quad \text{a.s.}\,,\\
\Var \big(\llangle  \mathcal{E}_\Omega   ,\bX^\top  (\Theta- \Ts)\bX  \rrangle_F \big) &\le \frac{1}{4}\| \bX^\top  (\Theta- \Ts)\bX  \|_{F,\Omega}^2\,.
\label{VarEmpProcBound}
\end{align}

\subsection{Entropy bounds for some classes of matrices}
Recall that an \emph{\(\eps\)-net} of a bounded subset \(\mathbf{K}\) of some metric space with a metric \(\rho\) is a collection \(\{K_1,...,K_{N_\eps} \} \in \mathbf{K}\)
such that for each \(K \in \mathbf{K}\), there exists \(i \in \{1,...,N_\eps\} \) such that \(\rho(K,K_i) \le \eps\). The \emph{\(\eps\)-covering 
	number} \(N(\eps, \mathbf{K}, \rho)\) is the cardinality of the smallest \(\eps\)-net. The \emph{\(\eps\)-entropy} of the class \(\mathbf{K}\) is 
defined by \(H(\eps, \mathbf{K}, \rho) = \log_2 N(\eps, \mathbf{K}, \rho)\). 
The following statement is adapted from Lemma 3.1 in \cite{candes2011tight}. 
\begin{lemma}
	\label{entrineq} Let \(\T_0 :=  \{\Theta \in \R^{k\times k}: \rank(\Theta) \le r, \| \Theta\|_F \le 1 \}\). Then it holds for any \(\eps > 0\)
	\[
	H(\eps, \T_0, \|\cdot \|_F) \le \big((2k + 1)r + 1\big) \log\big(\frac{9}{\eps}\big)\,.
	\]
\end{lemma}

\subsection{Proof of Theorem~\ref{ThmUpperBound} and Theorem~\ref{ThmPredUpperBound}}
\label{ProofUppBound}

It suffices to show the following uniform deviation inequality 
\begin{equation}
\sup_{\Ts \in  \cP_{k,r}(M)} 
\P_{\Ts}\big( \ell(\Ts) -\ell(\Th)  + p(\Th) > 2 p(\Ts) + R_t^2\big) \le \ex^{- c R_t}\,,
\label{ssPTsb}
\end{equation}
for any \(R_t > 0\) and some numeric constant \(c >0 \). Indeed, then taking \(R_t^2 = p(\Ts)\), it follows that \(\ell(\Ts) -\ell(\Th)  \le  3 p(\Ts)\) 
uniformly for all \(\Ts\) in the considered class with probability at least \( 1 - \ex^{- c \sqrt{\pen(\Ts)}}\). The upper bound
\eqref{rateUpperBound0} of Theorem~\ref{ThmUpperBound} the follows directly 
integrating the deviation inequality \eqref{ssPTsb}, while 
the upper bound on the prediction error in Theorem~\ref{ThmPredUpperBound} further follows using \eqref{TayMean} and the smoothness of the logistic function, \(\sup_{t\in \R}\sigma^\prime(t) \le 1/4 \). 
Define 
\begin{equation}
\label{tauG_R}
\tau^2(\Theta; \Ts) :=  \ell(\Ts) -\ell(\Theta)  + \pen(\Theta)\,,\,  
G_R(\Ts):= \{\Theta: \tau(\Theta; \Ts) \le R \}\,.	
\end{equation}
The inequality \eqref{ssPTsb} clearly holds on the event \(\{ \tau^2(\Th; \Ts) \le  2 \pen(\Ts)\}\).
In view of \(\Lc(\Th ) - \pen(\Th) \ge \Lc(\Ts) - \pen(\Ts)\), we have on the complement:
\[
\llangle  \mathcal{E}_\Omega   ,\bX^\top  (\Th- \Ts)\bX  \rrangle \ge \ell(\Ts) -\ell(\Th) + \pen(\Th) - \pen(\Ts) \ge 
\frac{1}{2} \tau^2(\Th; \Ts) \,.
\]
Therefore, for any \(\Ts \in \cP_{k,r}(M) \), we have 
\[
\P_{\Ts}\big(\tau^2(\Th; \Ts)> 2 \pen(\Ts) + R_t^2\big) 
\le \P_{\Ts}\Big(\sup_{\tau(\Theta; \Ts) \ge R_t} \frac{\llangle  \mathcal{E}_\Omega   ,\bX^\top  (\Theta- \Ts)\bX  \rrangle }{\tau^2(\Theta; \Ts)} \ge \frac{1}{2}\Big) \,.
\]
We now apply the so-called ``peeling device'' (or ``slicing'' as it sometimes called in the literature). The idea is to ``slice''
the set \(\tau(\Theta; \Ts) \ge R_t\) into pieces on which the penalty term \(\pen(\Theta)\) is fixed 
and the term \(\ell(\Ts) -\ell(\Theta) \) is bounded. It follows,
\begin{align}
&\P_{\Ts}\Big(\sup_{\tau(\Theta; \Ts) \ge R_t} \frac{\llangle  \mathcal{E}_\Omega   ,\bX^\top  (\Theta- \Ts)\bX  \rrangle }{\tau^2(\Theta; \Ts)} \ge \frac{1}{2}\Big) 
\nonumber	\\ &\quad \le 
\sum_{K = 1}^{d} \sum_{R = 1}^{K}\sum_{s = 1}^{\infty}\P_{\Ts}\Big(\sup_{\substack{\Theta \in G_{2^s R_t}(\Ts) \\k(\Theta)=K, \, \rank(\Theta) = R}} \llangle  \mathcal{E}_\Omega   ,\bX^\top  (\Theta- \Ts)\bX  \rrangle \ge \frac{1}{8}2^{2s}R_t^2\Big) \,.
\label{IneqDevMain}
\end{align}
On the set \(\{ \Theta \in G_{2^s R_t}(\Ts),\) \( k(\Theta)=K,  \rank(\Theta) = R\}\), it holds by the definitions \eqref{tauG_R}
\[
\ell(\Ts) -\ell(\Theta) \le  2^{2s}R_t^2 - \pen(K,R)\,,
\]
and therefore using \eqref{TayMean}, this implies 
\begin{equation}
\label{XTsThXGZ}
\| \bX^\top (\Ts - \Theta)\bX   \|_{F,\Omega} \le Z(K,R,s) \,, \quad  Z^2(K,R,s) = \frac{2}{\Lm }\big(2^{2s}R_t^2 - \pen(K,R)\big)\,.
\end{equation}
Let us fix the location of the block, that is the support of a matrix \(\Theta^\prime \in \G_1 := \{ \Theta \in \R^{d\times d}: k(\Theta) = K,\rank(\Theta) = R \}\) 
belongs to the upper-left block of size \(K \times K\).
Then following the lines of the proof of Lemma~\ref{entrineq} and using the singular value decomposition, we derive 
\[
H(\eps, \{\bX^\top \Theta^\prime \bX:\Theta^\prime  \in  \G_1, \|\bX^\top \Theta^\prime \bX \|_{F,\Omega} \le B  \}, \|\cdot \|_{F,\Omega}) \le \big((2K + 1)R + 1\big) \log\big(\frac{9B}{\eps}\big)\,.
\]
Consequently, for the set \(\T := \{  \bX^\top ( \Theta - \Ts) \bX : \Theta \in \R^{d\times d}, \rank(\Theta) = R, k(\Theta) = K, 	\| \bX^\top (\Ts - \Theta)\bX   \|_{F,\Omega} \le Z(K,R,s) \} \),
we obtain
\[
H(\eps, \T , \|\cdot \|_{F,\Omega}) \le \big((2K + 1)R + 1\big) \log\Big(\frac{9Z(K,R,s)}{\eps}\Big) + K \log\big(\frac{d\ex}{K}\big)\,.
\]
Denote \(t(K,R) := \sqrt{KR}  + \sqrt{K \log\big(\frac{d\ex}{K}\big)}\). 
By Dudley's entropy integral bound, see \cite{DUDLEY1967290} and \cite{gine2015mathematical} for a more recent reference, we then have 
\begin{align*}
\E\big[&\sup_{\substack{\Theta \in G_{2^s R_t}(\Ts) \\k(\Theta)=K, \, \rank(\Theta) = R}} \llangle  \mathcal{E}_\Omega   ,\bX^\top  (\Theta- \Ts)\bX  \rrangle  \big] 
\le C^\prime \int_{0}^{Z(K,R,s)} \sqrt{H(\eps, \T , \|\cdot \|_\Omega)}  \d\eps \nonumber \\
&\le  C^{\prime\prime} \sqrt{kr}  \int_{0}^{9Z(K,R,s)}  \sqrt{\log\Big(\frac{9Z(K,R,s)}{\eps}\Big) } \d\eps 
+ 9C^{\prime\prime} Z(K,R,s)\sqrt{K \log\big(\frac{d\ex}{K}\big)}\nonumber\\
&\le CZ(K,R,s)t(K,R) \,,
\end{align*}
for some universal constant \(C > 0\). 
Furthermore, by Bousquet's version of Talagrand's inequality, see Theorem~\ref{BousTalan}, in view of the bounds \eqref{EmpProcBound} and \eqref{VarEmpProcBound}, we have 
for all \(u > 0\)
\begin{align*}
\P_{\Ts}\Big(&\sup_{\substack{\Theta \in G_{2^s R_t}(\Ts) \\k(\Theta)=K, \, \rank(\Theta) = R}} \llangle  \mathcal{E}_\Omega   ,\bX^\top  (\Theta- \Ts)\bX  \rrangle \ge 
CZ(K,R,s)t(K,R) \nonumber \\ 
&+ \sqrt{\Big(\frac{1}{2}Z^2(K,R,s) + 4 C  Z^2(K,R,s)t(K,R) \Big)u } + \frac{Z(K,R,s)u}{3}  \Big) \le \ex^{-u}\,.
\end{align*}
Taking \(u(K,R,s) := \Lm^{1/2} Z(K,R,s) + \Lm^{-1/2} t(K,R) + 2\log d\) and using inequalities \(\sqrt{c_1 + c_2} \le \sqrt{c_1} + \sqrt{c_2}\) and \(\sqrt{c_1c_2}\le \frac{1}{2}(c_1\eps + \frac{c_2}{\eps})\), which hold for any \(c_1,c_2,\eps > 0 \), we obtain
\begin{align*}
\P_{\Ts}\Big(&\sup_{\substack{\Theta \in G_{2^s R_t}(\Ts) \\k(\Theta)=K, \, \rank(\Theta) = R}} \llangle  \mathcal{E}_\Omega   ,\bX^\top  (\Theta- \Ts)\bX  \rrangle \\& \ge 
\frac{1}{16}\Lm Z^2(K,R,s) + C_{1} ^2 t^2(K,R)/\Lm  \Big)\le \ex^{-u(K,R,s)}\,,
\end{align*}
for some numeric constant \(C_{1}> 0\).
Plugging this back into \eqref{IneqDevMain} and using \eqref{XTsThXGZ}, we obtain
\[
\P_{\Ts}\Big(\sup_{\substack{\Theta \in G_{2^s R_t}(\Ts) \\k(\Theta)=K, \, \rank(\Theta) = R}} \llangle  \mathcal{E}_\Omega   ,\bX^\top  (\Theta- \Ts)\bX  \rrangle \ge \frac{1}{8}2^{2s}R_t^2\Big) \le \ex^{-u(K,R,s)} \,, 
\]
for some numeric constant \(C_{2} > 0\), provided that 
\begin{equation}\label{eq:XTsThXewqeGZ}
\frac{1}{16}\Lm Z^2(K,R,s) + \frac{C_{1}^2}{\Lm} t^2(K,R) \le \frac{1}{8} 2^{2s}R_t^2 = \frac{1}{16} \Lm Z^2(K,R,s)  + 8 \pen(K,R)\,,
\end{equation}
which is satisfied for \(\pen(K,R) \ge (C_{1} ^2 / \Lm)  t^2(K,R)\).
Therefore,
\[
\P_{\Ts}\Big(\sup_{\tau(\Theta; \Ts) \ge R_t} \frac{\llangle  \mathcal{E}_\Omega   ,\bX^\top  (\Theta- \Ts)\bX \rrangle }{\tau^2(\Theta; \Ts)} \ge \frac{1}{2}\Big)  \le 
\sum_{K = 1}^{d} \sum_{R = 1}^{K}\sum_{s = 1}^{\infty} \ex^{-u(K,R,s)} 
\le \ex^{-c R_t}\,,
\]
for some numeric constants \(c > 0\) using \eqref{eq:XTsThXewqeGZ}, which concludes the proof.

The following prominent result is due to \cite{bousquet2002bennett}.
\begin{theorem}[Bousquet's version of Talagrand's inequality] 
	Let \((B,\mathcal{B})\) be a measurable space and let \(\varepsilon_1, ..., \varepsilon_n\) 
	be independent \(B\)-valued random variables. Let \(\mathcal{F}\) be a countable set of measurable 
	real-valued functions on \(B\) such that \(f(\varepsilon_i)\le b < \infty\) a.s. and 
	\(\E f(\varepsilon_i) = 0\) for all \(i = 1,...,n\),  \(f \in \mathcal{F} \) . Let  
	\[
	S := \sup_{f \in \mathcal{F}}\sum_{i=1}^{n}f(\varepsilon_i)\,, \quad \quad   v:= \sup_{f \in \mathcal{F}}\sum_{i=1}^{n} \E [f^2(\varepsilon_i)] . 
	\]
	Then for all \(u > 0\), it holds that
	\begin{equation}
	\label{BousTalan}
	\P\Big(S - \E[S] \ge\sqrt{2(v + 2b\E[S])u} +  \frac{bu}{3}\Big) \le \ex^{-u}\,.
	\end{equation}

\end{theorem}

\subsection{Proof of Theorem~\ref{ThmConstrainedMLE}}
For the MLE \(\Th_{k,r}\), it clearly holds \( \Lc(\Th_{k,r} )  \ge \Lc(\Ts)\) implying 
\begin{equation*}
\ell(\Ts) - \ell(\Th_{k,r}) \le 	\llangle  \mathcal{E}_\Omega   ,\bX^\top  (\Th_{k,r}- \Ts)\bX  \rrangle .
\end{equation*} 
Furthermore, in view of \eqref{TayMean}, we derive 
\begin{align}
\frac{\Lm}{2} \|\bX^\top(\Th_{k,r}- \Ts ) \bX  \|_{F, \Omega}& \le 
\frac{	\llangle  \mathcal{E}_\Omega   ,\bX^\top  (\Th_{k,r}- \Ts)\bX  \rrangle}{\|\bX^\top( \Th_{k,r}- \Ts ) \bX  \|_{F, \Omega}} \\
&\le \sup_{\Theta  \in  \cP_{k,r}(M)} 
\frac{	\llangle  \mathcal{E}_\Omega   ,\bX^\top  (\Theta- \Ts)\bX  \rrangle}{\|\bX^\top( \Theta- \Ts ) \bX  \|_{F, \Omega}}\,.
\label{eq:errvO}
\end{align}
Following the lines of Section~\ref{ProofUppBound}, by Dudley's integral  we next obtain 
\begin{equation*}
\E \Big(  \sup_{\Theta  \in  \cP_{k,r}(M)} 
\frac{	\llangle  \mathcal{E}_\Omega   ,\bX^\top  (\Theta- \Ts)\bX  \rrangle}{\|\bX^\top( \Theta- \Ts ) \bX  \|_{F, \Omega}} \Big) \le c\sqrt{kr} + c \sqrt{k \log\big(\frac{d\ex}{k}\big)}\,,
\end{equation*}
for some universal constant \(c > 0\). Plugging this bound back into \eqref{eq:errvO} and using the block isometry property yields the desired assertion.

\subsection{Proof of Theorem~\ref{LowerBoundThm}}
\begin{proof}
	The proof is split into two parts. First, we show a lower bound of the order \(kr\) and then a lower bound of the order \(k \log(d\ex/k)\). 
	A simple inequality 
	\( (a + b)/2 \le \max\{a,b\}\) for all \(a, b > 0\) then completes the proof. Both parts 
	of the proof exploit a version of remarkable Fano's inequality given in Proposition~\ref{FanosMethod} to follow, 
	cf. Section 2.7.1 in  \cite{Tsybakov2008}. 
	
	\emph{1. A bound \(kr\).} 
	The proof of this bound is similar to the proof of a minimax lower bound 
	for estimating a low-rank matrix in the trace-norm regression model given in Theorem~5 in \cite{koltchinskii2011}. For the sake of completeness, we 
	provide the details here. Consider a subclass of matrices
	\begin{align*}
	\mathcal{C} &= \left \{
	\begin{tabular}{c}
	\(A \in \R^{k\times  r }:  a_{i,j} = \{0, \alpha_N \}, 1\le i \le k , 1 \le j \le  r\) 
	\end{tabular} 
	\right \}\,,\\ \alpha_N^2 &=   \frac{\gamma \log 2}{ 1 + \Delta_{\Omega,2k}(\bX)} \frac{r}{2kN}  \,,
	\end{align*}
	where \(\gamma > 0\) is a positive constant, \(\Delta_{\Omega,2k}(\bX) > 0\) is the the block isometry  constant from 
	Definition~\ref{AssWellCond} and  \(\lfloor x \rfloor\) denotes the integer part of \(x\).
	Further define 
	\[
	\mathcal{B}(\mathcal{C}) = \big\{\frac{1}{2} (A + A^\top):  A  = (\tilde{A}| \cdots | \tilde{A}| O) \in \R^{k\times k}, \tilde{A} \in \mathcal{C} \big\} \,,
	\]
	where \(O\) denotes the \(k \times (k - r  \lfloor k  / r \rfloor)\) zero matrix. By construction, 
	any matrix \(\Theta \in \mathcal{B}(\mathcal{C})\) is symmetric, has rank at most \(r\) with entries bounded by \(\alpha_N\). 
	Applying a standard version of the Varshamov-Gilbert lemma, see Lemma 2.9 in \cite{Tsybakov2008}, 
	there exists a subset \(\mathcal{B}^\circ  \subset \mathcal{B}(\mathcal{C})\) of cardinality \(\card(\mathcal{B}^\circ ) \ge 2^{kr/16}  + 1\) such that 
	\[	
	\frac{kr}{16} \big( \frac{\alpha_N}{2}\big)^2  \big\lfloor \frac{k}{r} \big\rfloor 
	\le \|\Theta_u - \Theta_v \|_F^2 \le k^2 \alpha_N^2 \,,
	\]
	for all \(\Theta_u, \Theta_v \in \mathcal{B}^\circ\). 
	Thus \(\mathcal{B}^\circ\)
	is a \(2\delta\)-separated set in the Frobenius metric with \(\delta^2 = \frac{kr}{64} \big( \frac{\alpha_N}{2}\big)^2\lfloor \frac{k}{r} \rfloor\).
	The Kullback-Leibler divergence between the measures \(\P_{\Theta_{u}}\) and \(\P_{\Theta_{v}}\),
	\(\Theta_{u}, \Theta_{v} \in \mathcal{B}^\circ\), \(u\neq v\),
	is upper bounded as
	\begin{align*}
	\Kc(\P_{\Theta_{u}},\P_{\Theta_{v}}) & = \E_{\P_{\Theta_{u}}} [ \Lc(\Theta_{u})] - \E_{\P_{\Theta_{u}}} [ \Lc(\Theta_{v})] 
	\le \frac{1}{8} \sum_{(i,j) \in \Omega} \llangle X_j X_i^\top, \Theta_u - \Theta_v  \rrangle_F^2  \\& \le \frac{1+ \Delta_{\Omega,2k}(\bX)}{8}  k^2\alpha_N^2 N \,.
	\end{align*}
	Taking \(\gamma > 0\) small enough, we obtain 
	\[
	\frac{1+ \Delta_{\Omega,2k}(\bX)}{8}  k^2\alpha_N^2N +\log 2 =  
	\frac{kr}{16} \gamma \log 2  +\log 2 = \log(2^{ \frac{kr}{16}\gamma + 1})
	< \log (2^{kr/16}  + 1)\,,
	\]
	which, in view of Proposition~\ref{FanosMethod}, yields the desired lower bound.

	\emph{2. A bound \(k \log(d\ex/k)\).}
	Let \(K = {d \choose k } \) 
	and consider the set \(\G_k^{\alpha_N} \subset \cP_{k,1}(M)\)  from the reduction scheme in Section~\ref{ReductionSect} with 
	\[	
	\alpha_N^2 =  \frac{4\gamma\log 2}{kN(1+\Delta_{\Omega,2k}(\bX))} \log\big(\frac{d\ex}{k}\big)\,,
	\]
	where \(\gamma  > 0\) is a positive constant. 
	Using simple calculations, we then have 
	\((2k -1)\alpha_N^2 \le \|\Theta_u - \Theta_v \|_F^2 \le 2 k^2\alpha_N^2\)
	for all \(\Theta_u, \Theta_v \in \G_k^{\alpha_N}\), \(u\neq v\). Furthermore, according to Lemma~\ref{VGlemma} to follow, there exists a subset 
	\(\G^{\alpha_N, 0}_{k} \subset \G_k^{\alpha_N}\)  such that 
	\[
	c_0 k^2\alpha_N^2 \le \|\Theta_u - \Theta_v \|_F^2 \le 2 k^2 \alpha_N^2\,,
	\]
	and of cardinality \( \card(\G^{\alpha_N, 0}_{k}) \ge  2^{\rho k \log(d\ex/k)} + 1\)  for some \(\rho > 0\) depending on a constant \(c_0 > 0\)
	and independent of \(k\) and \(d\). Thus \(\G^{\alpha_N, 0}_{k}\)
	is a \(2\delta\)-separated set in the Frobenius metric with \(\delta^2 = c_0k^2\alpha_N^2/4\).
	The Kullback-Leibler divergence between the measures \(\P_{\Theta_{u}}\) and \(\P_{\Theta_{v}}\),
	\(\Theta_{u}, \Theta_{v} \in \G^{\alpha_N, 0}_{k}\), \(u\neq v\),
	is upper bounded as
	\begin{align*}
	\Kc(\P_{\Theta_{u}},\P_{\Theta_{v}}) & = \E_{\P_{\Theta_{u}}} [ \Lc(\Theta_{u})] - \E_{\P_{\Theta_{u}}} [ \Lc(\Theta_{v})] 
	\le \frac{1}{8} \sum_{(i,j)\in \Omega} \llangle X_j X_i^\top, \Theta_u - \Theta_v  \rrangle_F^2 \\
	& \le \frac{1+ \Delta_{\Omega,2k}(\bX)}{4}  k^2\alpha_N^2 N \,,
	\end{align*}
	for all \(u \neq v\) and \(\Delta_{\Omega,2k}(\bX) >0 \) from Definition~\ref{AssWellCond}. 
	As in the first part of the proof, taking \(\gamma > 0\) small enough, we obtain
	\begin{align*}
	\frac{1+ \Delta_{\Omega,2k}(\bX)}{4}  k^2\alpha_N^2N +\log 2  
	&= k \gamma \log(2) \log\big(\frac{d\ex}{k}\big) +\log 2  
	= \log(2^{k \gamma \log(d\ex/k) + 1})\\
	&<  \log(2^{\rho k \log(d\ex/k)} + 1) \,.
	\end{align*}
	The desired lower bound then follows from Proposition~\ref{FanosMethod}.
\end{proof}

\begin{proposition}[Fano's method]
	\label{FanosMethod}
	Let \(\{\Theta_1,...,\Theta_J\}\) be a \(2\delta\)-separated set in \(\R^{d\times d}\) in the Frobenius metric, 
	meaning that \(\| \Theta_k - \Theta_l\|_F \ge 2\delta\) for all elements \(\Theta_k, \Theta_l\), \(l \neq k\) in the set. 
	Then for any increasing and measurable function \(F: [0, \infty) \to [0, \infty)\), the minimax risk is lower bounded as 
	\[
	\label{FanoLowerBound}
	\inf_{\Th} \sup_{\Theta} \E_{\P_\Theta} \big[F(\|\Th - \Theta\|_F) \big] \ge 
	F(\delta) \big(1 - \frac{\sum_{u,v} \Kc(\P_{\Theta_{u}}, \P_{\Theta_{v}})/J^2  + \log 2}{\log J} \big)\,.
	\]
\end{proposition}

\begin{lemma}[Variant of the Varshamov-Gilbert lemma]
	\label{VGlemma}
	Let \(\G \subset  \cP_{k,1}(M)\) be a set of \(\{0,1\}^{d\times d}\) symmetric block-sparse matrices with the size of the block \(k\), where \(k \le \alpha \beta d\) 
	for some \(\alpha,\beta \in (0,1)\). Denote 
	\(K = {d \choose k}\) the cardinality of \(\G\) and \(\rho_H(E, E^\prime) = \sum_{i,j} \Ind(E_{i,j} \neq E_{i,j}^\prime)\) the Hamming 
	distance between two matrices \(E, E^\prime \in \G\).
	Then there exists a subset \(\G^0 = \{E^{(0)},..., E^{(J)}\} \subset \G\) of cardinality 
	\[
	\log J := \log(\card(\G^0 ))\ge \rho k \log (\frac{d\ex}{k})\,,
	\]
	where \(\rho = \frac{\alpha}{- \log(\alpha\beta)}(- \log\beta +\beta  -1 )\)
	such that  
	\[
	\rho_H(E^{(k)}, E^{(l)}) \ge ck^2\,,
	\]
	for all \(k \neq l\) where  \(c = 2(1 - \alpha^2) \in (0,2)\).
\end{lemma}

\begin{proof} Let \(E^{(0)} = \{0\}^{k\times k}\), \(D = ck^2\), and construct the set \(\mathcal{E}_{1} = \{E \in \G : \rho_H(E^{(0)}, E) >  D\} \).
	Next, pick any \(E^{(1)} \in \mathcal{E}_{1}\) and proceed iteratively so that for a matrix  \(E^{(j)} \in \mathcal{E}_{j}\) we construct the set
\[	
		\mathcal{E}_{j+1} = \{E \in \mathcal{E}_{j} : \rho(E^{(j)}, E) >  D \}\,.
\]
	Let \(J\) denote the last index \(j\) for which \(\mathcal{E}_{j} \neq \varnothing\). It remains to bound the cardinality 
	\(J\) of the constructed set \(\G^0 = \{E^{(0)},..., E^{(J)}\}\). For this, we consider the 
	cardinality \(n_j \) of the subset \(\{\mathcal{E}_{j} \setminus \mathcal{E}_{j+1} \}\):
\[
		\label{njEcj}
		n_j := \#\{\mathcal{E}_{j} \setminus \mathcal{E}_{j+1} \} \le \#\{E \in \G : \rho_H(E^{(j)}, E) \le  D \}.
\]
	For all \(E, E^\prime \in \G \), we have 
\[	
		\label{rHEEp2}
		\rho_H(E, E^\prime) = 2(k^2 - (k - m)^2)\,, 
\]
	where \(m \in [0, k]\) corresponds to the number of distinct columns of \(E\) (or \(E^\prime\)). Solving the quadratic equation \eqref{rHEEp2}
	for  \(\rho_H(E, E^\prime) = D = ck^2\) we obtain 
\[	
		m_D = k (1 - \sqrt{1- c/2})\,,
\]
	for the maximum number of distinct columns of a block-sparse matrix \(E\) (and \(E^\prime\)) such that 
	\(\rho_H(E, E^\prime) \le D = ck^2\) for  \(c \in [0,2]\). For instance, in order to get the distance between matrices \(2k^2\), i.e. \(c = 2\)
	we need to shift all the \(k\) columns (and consequently rows) and so the number of distinct columns of a matrix is \(m = k\), and in order to get the minimal possible distance 
	\(4k - 2\), i.e. \(c = (4k - 2)/k^2\) we need to shift only one column and a corresponding row, i.e. \(m = 1\). Therefore, for \(n_j\) 
	in \eqref{njEcj}, we have 
\[	
		n_j \le  \#\{E \in \G : \rho_H(E^{(j)}, E) \le  D \} =\sum_{i = 0}^{m_D} {k \choose i} {d - k \choose  i }
		=\sum_{i = k - m_D}^{k} {k \choose i} {d - k \choose k - i }\,.
\]
	Together with an evident equality \(\sum_{j =0}^J n_j = K = {d \choose k}\), this implies 
\[
		\label{sikmD}
		\sum_{i = k - m_D}^{k} {k \choose i} {d- k \choose k - i } / {d \choose k} \ge \frac{1}{J + 1}\,.
 \]
	Note that taking \(m_D = k\), which as we have seen corresponds to \(c = 2\), we have a trivial bound \(J + 1 \ge 1\) using 
	Vandermonde's convolution. Furthermore, the expression on the left-hand side in \eqref{sikmD}
	is exactly the probability \(\P(X \ge k - m_D) = \P (X \ge k\alpha)\) for \(\alpha = \sqrt{1- c/2}\), where the variable \(X\) follows the hypergeometric distribution \(H(d, k, k/d)\). The rest of the proof is based on applying  Chernoff's inequality and follows the scheme of the proof 
	of Lemma 4.10 in \cite{massart2007concentration}.
\end{proof}

\end{document}